\documentclass[reqno]{amsart}

\usepackage{amsmath, mathrsfs, url}
\usepackage[a4paper]{geometry}

\newtheorem{thm}{Theorem}[section]
\newtheorem{cor}{Corollary}

\newtheorem{lem}[thm]{Lemma}

\newtheorem{remark}{Remark}

\newtheorem{assumptionA}{A-\hspace{-1.2mm}}

\title[Milstein-type Schemes]{Milstein-type Schemes of SDE Driven by L\'evy Noise with Super-linear Diffusion Coefficients}

\author[C. Kumar]{Chaman Kumar \\ Indian Institute of Technology Roorkee}

\thanks{This research was carried out while the author was visiting School of Mathematics, University of Edinburgh, United Kingdom during June-July, 2017. The author would like to thank the school for the hospitality and support which helped in timely finishing the project.}

\begin{document}
\maketitle
\begin{abstract}
We present a Milstein-type scheme for stochastic differential equations driven by L\'evy noise with super-linear diffusion coefficients and establish its strong convergence. 
\end{abstract}

\section{Introduction}

Let $T>0$ be a fixed constant and $(\Omega, \{\mathscr{F}_t\}_{t \in [0,T]}, \mathscr{F}, P)$  be a right continuous complete filtered probability space. Suppose that $w$ is an ${\mathbb{R}}^m-$valued  standard Wiener process. Further, assume that $N(dt,dz)$ is a Poisson random measure on finite measure space $(Z,\mathscr{Z}, \nu)$ with intensity $\nu$ (i.e. $\nu(Z)<\infty$) independent of $w$. Define $\tilde N(dt,dz):=N(dt,dz)-\nu(dz)dt$. Suppose that $b(x)$ and $\sigma(x)$ are  $\mathscr{B}(\mathbb R^d)$-measurable functions and take values in ${\mathbb{R}}^d$ and ${\mathbb{R}}^{d\times m}$ respectively. Moreover, $\gamma(x,z)$ is a $\mathscr{B}(\mathbb{R}^d)\otimes \mathscr{Z}$-measurable function with values in $\mathbb{R}^d$. The functions $b(x)$, $\sigma(x)$ and $\gamma(x,z)$ are assumed to be twice differentiable in $x\in \mathbb{R}^d$. 

Consider a $d$-dimensional SDE given by,
\begin{align} \label{eq:sde}
x_t=& \xi + \int_{0}^t b(x_s)ds+\int_{0}^t \sigma(x_s)dw_s+ \int_{0}^t\int_{Z}\gamma(x_s,z)\tilde N(ds,dz),
\end{align}
almost surely for any $t \in [0,T]$ with initial value $\xi\in \mathbb{R}^d$ as an $\mathscr{F}_{0}$-measurable  random variable.
\begin{remark}
For the ease of notation, we write the right hand side of equation \eqref{eq:sde} with $x_t$ in place of $x_{t-}$. This notational convenince has been following from  \cite{dareiotis2016, kumar2017, kumar2016b}. 
\end{remark}

In this paper, we are interested in approximating SDE \eqref{eq:sde} in strong $\mathcal{L}^2$-sense when both drift and diffusion are allowed to grow super-linearly. FOr super-linear coefficients, recently many results have been developed, for example, \cite{hutzenthaler2012, hutzenthaler2013, kumar2016b, sabanis2013, kumar2016a, tretyakov2013, wang2013, zhang2014}. To the best of author's knowledge, this is the first result on explicit Milstein scheme for SDE driven by L\'evy noise with super-linear drift and diffusion coefficients. 

Let us introduce some notations. Let $v^i$ be the $i$th element of a vector $v \in \mathbb{R}^d$. For a matrix $a\in \mathbb{R}^{d \times m}$,  $a^j$ and $a^{ij}$ stand for its $j$th column and $ij$th element respectively. For a function $f:\mathbb{R}^d\to \mathbb{R}^d$, 
 $\Delta f(\cdot)$ denotes a $d\times d$ matrix with $ij$-th entry as $[\Delta f(\cdot)]^{ij}:=\partial f^i(\cdot)/\partial x^j$.  Similarly, for every $k=1,\ldots,m$ and  a function $g:\mathbb{R}^d \to \mathbb{R}^{d \times m}$,  $\Lambda^kg(\cdot)$ stands for a $d \times m$ matrix with $ij$-th element given by $[\Lambda^k g(\cdot)]^{ij}:= \sum_{l=1}^{d} g^{lk}(\cdot) \frac{\partial g^{ij}(\cdot)}{\partial x^l}$. Further, same notation $|\cdot|$ is used to denote the Euclidean norm of a vector and the Hilbert-Schmidt norm of matrix, which can be understood from the context where it appears without causing any confusion. Also, $a^*$ stands for transpose of $a \in \mathbb{R}^{d \times m}$ and $vu$ for inner product of $v,u \in \mathbb{R}^d$. $\lfloor r \rfloor$ denotes integer part of a real number $r$. Moreover, throughout this text,  $K$ is used to represent a generic constant whose value can differ at different occurrences.

Let $p_0 \geq 2$ and $\rho \geq 1$ be fixed constants. 

\begin{assumptionA} \label{as:sde:ini}
$E|\xi|^{p_0} < \infty$.
\end{assumptionA}

\begin{assumptionA} \label{as:sde:growth}
There exists a constant $L>0$ such that
$$
2xb(x)+(p_0-1)|\sigma(x)|^2  \leq L(1+|x|)^2
$$
for any $x \in \mathbb{R}^d$.
\end{assumptionA}

\begin{assumptionA} \label{as:sde:growth:gam}
There exists a constant $L>0$ such that
$$
\int_Z |\gamma(x,z)|^{p_0}\nu(dz) \leq L (1+|x|)^{p_0}
$$
for any $x \in \mathbb{R}^d$.
\end{assumptionA}

\begin{assumptionA} \label{as:sde:lipschitz}
There exists a constant $L>0$ and $\eta\in(1,\infty)$ such that
$$
\big\{2(x-\bar{x})(b(x)-b(\bar{x}))+ \eta |\sigma(x)-\sigma(\bar{ x})|^2 \big\} \vee \int_Z|\gamma(x,z)-\gamma(\bar{x},z)|^2\nu(dz) \leq L|x-\bar{x}|^2
$$
for any $x, \bar{x} \in \mathbb{R}^d$.
\end{assumptionA}

\begin{assumptionA} \label{as:sde:lip:der}
There exists a constant $L>0$ such that
\begin{align*}
|\Delta b(x)-\Delta b(\bar{x})| & \leq L(1+|x|+|\bar{x}|)^{\rho-1}|x-\bar{x}|
\\
|\Delta\sigma^{k}(x)-\Delta\sigma^{k}(\bar{x})| & \leq L(1+|x|+|\bar{x}|)^{\frac{\rho-2}{2}}|x-\bar{x}|, \,\,k=1,\ldots,m
\\
\int_Z |\Delta\gamma(x,z)-\Delta\gamma(\bar{x},z)|^2 \nu(dz)&\leq L|x-\bar{x}|^2
\end{align*}
for any $x, \bar{x} \in \mathbb{R}^d$. 
\end{assumptionA}

Before giving the explicit form of the Milstein-type scheme considered in this paper, let us introduce some notations. For every $n\in\mathbb{N}$,  define,
\begin{align*}
b^{(n)}(x)&:=\frac{b(x)}{1+n^{-1}|x|^{4\rho}} 
\\
\Lambda^{(n)}_0 (x)  & :=\frac{\sigma (x) }{1+n^{-1}|x|^{4\rho}} 
\\
\Lambda^{(n),k}_1 (x)  & :=\frac{\sum_{u=1}^{d} \sigma^{uk}(x)\frac{\partial \sigma(x)}{\partial x^u} }{1+n^{-1}|x|^{4\rho}}, \, \, k=1,\ldots, m 
\\
\Lambda^{(n)}_2 (x,z) & =\frac{\sum_{u=1}^{d} \gamma^{u}(x, z)\frac{\partial \sigma(x)}{\partial x^u}}{1+n^{-1}|x|^{4\rho}}, \, \, z\in Z 
\\
\Lambda^{(n)}_3 (x,z)&:=\frac{\sigma(x+\gamma(x,z))-\sigma(x)-\sum_{u=1}^{d} \gamma^{u}(x, z)\frac{\partial \sigma(x)}{\partial x^u}}{1+n^{-1}|x|^{4\rho}}, \, \, z\in Z
\end{align*}
for any $x\in \mathbb{R}^d$ and then write,  
\begin{align}
\sigma^{(n)}(t,x)=\Lambda^{(n)}_0 (x)+\int_{\kappa(n,t)}^t & \sum_{k=1}^{m} \Lambda_{1}^{(n),k}(x) dw_s^k+\int_{\kappa(n,t)}^t \int_Z \Lambda_2^{(n)}(x,z)\tilde{N}(ds,dz) \notag
\\
&+\int_{\kappa(n,t)}^t \int_Z \Lambda_3^{(n)}(x,z)N(ds,dz) \label{eq:sigma(n)}
\end{align}
almost surely for any $n \in \mathbb{N}$, $x \in \mathbb{R}^d$ and $t\in [0,T]$. Also, for every $n\in\mathbb{N}$, define,  
\begin{align*}
\Gamma_1^{(n),k}(x,z)&=\frac{\sum_{u=1}^{d} \sigma^{uk}(x)\frac{\partial \gamma(x,z)}{\partial x^u}}{1+n^{-1}|x|^{4\rho}}, k=1,\ldots,m, \mbox{ and } n\in \mathbb{N}, 
\\
\Gamma_2(x,z,z_1)&:=\sum_{u=1}^{d} \gamma^{u}(x,z_1)\frac{\partial \gamma(x,z)}{\partial x^u},\,\,  z_1 \in Z
\\
\Gamma_3(x,z, z_1)&:= \gamma(x+\gamma(x,z_1),z)-\gamma(x,z)-\Gamma_2(x,z,z_1), \, \, z_1\in Z
\end{align*}  
for any $x\in\mathbb{R}^d$, and $z\in Z$ and then write,
\begin{align}
\gamma^{(n)}(t,x,z):=\gamma(x,z)+\int_{\kappa(n,t)}^t& \sum_{k=1}^m \Gamma_1^{(n),k}(x,z) dw_s^k+\int_{\kappa(n,t)}^t\int_Z\Gamma_2(x,z,z_1)\tilde{N}(ds,dz_1) \notag
\\
&+\int_{\kappa(n,t)}^t\int_Z\Gamma_3(x,z,z_1)N(ds,dz_1) \label{eq:gamma(n)}
\end{align}
almost surely for any $x\in\mathbb{R}^d$, $n\in\mathbb{N}$, $z\in Z$  and $t\in [0,T]$.

In this article, following Milstein-type scheme of SDE \eqref{eq:sde} is considered, 
\begin{align} \label{eq:scheme} 
x_t^n=\xi+ \int_{0}^t b^{(n)}(x_{\kappa(n,s)}^n)ds+\int_{0}^t \sigma^{(n)}(s, x_{\kappa(n,s)}^n)dw_s+ \int_{0}^t\int_{Z}\gamma^{(n)}(s, x_{\kappa(n,s)}^n,z)\tilde N(ds,dz)
\end{align} 
almost surely for any $n \in \mathbb{N}$ and $t \in [0,T]$. 

The main result of this article is stated in the following theorem  which is proved in Section \ref{sec:rate} by a sufficiently large value of $p_0$ so that the required moment estimates are available in the following calculations. 
\begin{thm} \label{thm:main:thm}
Let Assumptions A-1 to  A-5 be satisfied. Then, the Milstein-type scheme \eqref{eq:scheme} converges to the true solution of SDE \eqref{eq:sde} in $\mathcal{L}^2$  with a rate of convergence given by, 
$$
\sup_{0 \leq t \leq T}E|x_t-x_t^n|^2 \leq Kn^{-1-\frac{2}{2+\delta}}+Kn^{-\frac{3}{2}-\frac{1}{2+\delta}} 
$$
for any $t\in[0,T]$ and $n\in\mathbb{N}$.
\end{thm}
\begin{remark} \label{rem:rate} 
If A-\ref{as:sde:ini} holds for all $p_0\geq 2$, then $\delta$ which appears in Theorem \ref{thm:main:thm} can be taken to be arbitrarily close to zero and hence one can achieve a rate of convergence arbitrarily close to one for the Milstein-type scheme defined in equation \eqref{eq:scheme}. This finding is consistent with the corresponding results of the classical Milstein scheme \cite{platen2010} and the Milstein-type scheme \cite{kumar2017}. Further, no such reduction in the rate of convergence of the Milstein-type scheme is observed when $\gamma\equiv 0$ which has been proved in  \cite{kumar2016a} i.e. when $\gamma\equiv 0$, then one achieves a rate of convergence equal to one for the convergence in $\mathcal{L}^p$ for any $p\geq 2$. 
\end{remark}

\begin{remark}
To implement the Milstein-type scheme defined in equation \eqref{eq:scheme}, one requires commutative conditions on the coefficients of the SDE \eqref{eq:sde}. These conditions are explained in \cite{platen2010}.
\end{remark}

\begin{remark}
The techniques developed here can be extended to higher order schemes. 
\end{remark}


The following useful remarks are required to prove the moment bound and the rate of convergence  of the Milstein-type scheme \eqref{eq:scheme}.
\begin{remark} \label{rem:growth}
By using Assumptions A-\ref{as:sde:lipschitz} and A-\ref{as:sde:lip:der}, there exists a constant $K>0$ such that 
\begin{align*}
\Big|\frac{\partial^2 b(x)}{\partial x^{u}x^{v}}\Big| & \leq K(1+|x|)^{\rho-1}, \,\, u,v=1,\ldots,d
\\
\Big|\frac{\partial^2 \sigma^k(x)}{\partial x^{u}x^{v}}\Big| & \leq K(1+|x|)^{\frac{\rho}{2}-1}, \,\, u,v=1,\ldots,d,k=1,\ldots,m
\\
\int_Z\Big|\frac{\partial^2 \gamma(x,z)}{\partial x^{u}x^{v}}\Big|^2\nu(dz) & \leq K, \,\, u,v=1,\ldots,d,
\\
|\Delta b(x)| & \leq K (1+|x|)^\rho, 
\\
|\Delta \sigma^{k}(x)| & \leq K(1+|x|)^\frac{\rho}{2}, \,\, k=1,\ldots,m,
\\
\int_Z |\Delta \gamma(x,z)|^2 \nu(dz) & \leq K 
\\
|b(x)- b(\bar{x})| & \leq K(1+|x|+|\bar{x}|)^{\rho}|x-\bar{x}|
\\
|\sigma^{k}(x)-\sigma^{k}(\bar{x})| & \leq K(1+|x|+|\bar{x}|)^\frac{\rho}{2}|x-\bar{x}|, \,\,k=1,\ldots,m
\\
|b(x)| & \leq K(1+|x|)^{\rho+1}
\\
|\sigma^k(x)| & \leq K(1+|x|)^{\frac{\rho}{2}+1}, \,\,k=1,\ldots,m
\\
 \int_Z |\gamma(x,z)|^2 \nu(dz) &\leq K(1+|x|)^2 
\end{align*}
for any $x, \bar{x}\in \mathbb{R}^d$, where the positive constant $K$  depends on $L$ appearing in assumptions A-\ref{as:sde:growth} to A-\ref{as:sde:lip:der}.
\end{remark}
\begin{remark} \label{rem:growth(n)}
It can be easily verified that,   
\begin{align*}
|b^{(n)}(x)|^{p} & \leq \min( K n^{p/4} (1+|x|)^p, K(1+|x|)^{(\rho+1)p} )
\\
|\Lambda_0^{(n)}(x)|^{p} &\leq \min(Kn^{p/8}(1+|x|)^{p}, K(1+|x|)^{(\rho+2)p})
\\
|\Lambda_1^{(n),k}(x)|^{p} & \leq \min(K n^{p/8}(1+|x|)^{p},K(1+|x|)^{(2\rho+2)p}), k=1,\ldots,m
\\
\int_Z|\Lambda_2^{(n)}(x,z)|^{p}\nu(dz) & \leq \min( Kn^{p/8}(1+|x|)^{p}, K(1+|x|)^{(\rho+2)p})
\\
\int_Z|\Lambda_3^{(n)}(x,z)|^{p}\nu(dz) & \leq \min( Kn^{p/8}(1+|x|)^{p}, K(1+|x|)^{(\rho+2)p})
\\
\int_Z|\Gamma_1(x,z)|^{p}\nu(dz) & \leq \min(Kn^{p/8}(1+|x|)^{p},K(1+|x|)^{(\rho+2)p})
\\
\int_Z\int_Z|\Gamma_2(x,z,z_1)|^{p}\nu(dz_1)\nu(dz) & \leq K(1+|x|)^{p}
\\
\int_Z\int_Z|\Gamma_3(x,z,z_1)|^{p}\nu(dz_1)\nu(dz) & \leq K(1+|x|)^{p}
\end{align*}
for any $2\leq p\leq p_0$, $x\in\mathbb{R}^d$, and $n\in\mathbb{N}$ where positive constant $K$ does not depend on $n$. 
\end{remark}
\section{Moment Bounds}
For the proof of the following lemma, one can refer to \cite{kumar2016b,  situ2005}.
\begin{lem} \label{lem:sde:moment:bound}
Let Assumptions A-\ref{as:sde:ini} to A-\ref{as:sde:lipschitz} be satisfied.  Then, there exists a unique solution to the SDE \eqref{eq:sde} and
$$
\sup_{0 \leq t \leq T}E|x_t|^{p_0} \leq K,
$$
where $K:=K(L,T,p_0,m, d, E|\xi|^{p_0})$ is a positive constant.
\end{lem}
\begin{lem} \label{lem:sig(n):c}
Let Assumptions A-\ref{as:sde:ini} to A-\ref{as:sde:lip:der} be satisfied. Then, the following holds, 
\begin{align*}
E\big(|\sigma^{(n)}(t,x_{\kappa(n,t)}^n|^p |\mathscr{F}_{\kappa(n,t)}\big) \leq Kn^{p/8}(1+|x_{\kappa(n,t)}|)^p
\end{align*}
almost surely for any $2\leq  p \leq p_0$, $t\in[0,T]$ and $n\in\mathbb{N}$ where positive constant $K$ does not depend on $n$. 
\end{lem}
\begin{proof} By using the definition of $\sigma^{(n)}$ given in equation \eqref{eq:sigma(n)}, 
\begin{align}
E\big(|\sigma^{(n)}(t,x_{\kappa(n,t)}^n)|^{p}\big|\mathscr{F}_{\kappa(n,t)}\big) & \leq K E\big(|\Lambda^{(n)}_0 (x_{\kappa(n,t)}^n)|^p\big|\mathscr{F}_{\kappa(n,t)}\big)\notag  
\\
& + K E\Big(\Big|\int_{\kappa(n,t)}^t   \sum_{k=1}^{d} \Lambda_{1}^{(n),k}(x_{\kappa(n,s)}^n) dw_s^k\Big|^p \Big|\mathscr{F}_{\kappa(n,t)}\Big) \notag
\\
&+KE\Big( \Big|\int_{\kappa(n,t)}^t  \int_Z \Lambda_2^{(n)}(x_{\kappa(n,s)}^n,z)\tilde{N}(ds,dz)\Big|^p \Big|\mathscr{F}_{\kappa(n,t)}\Big) \notag
\\
& +K E\Big( \Big|\int_{\kappa(n,t)}^t  \int_Z \Lambda_3^{(n)}(x_{\kappa(n,s)}^n,z)N(ds,dz)\Big|^p \Big|\mathscr{F}_{\kappa(n,t)}\Big) \notag
\end{align}
which on the application of an elementary inequality of stochastic integral and H\"{o}lder's inequality gives, 
\begin{align}
E\big(|\sigma^{(n)}(t,x_{\kappa(n,t)}^n)|^{p}\big|\mathscr{F}_{\kappa(n,t)}\big) & \leq Kn^{p/8} (1+|x_{\kappa(n,t)}^n|)^p  \notag
\\
&+ Kn^{-\frac{p}{2}+1} E\Big(\int_{\kappa(n,t)}^t  \Big| \sum_{k=1}^{d} \Lambda_{1}^{(n),k}(x_{\kappa(n,s)}^n)\Big|^p ds \Big|\mathscr{F}_{\kappa(n,t)}\Big) \notag
\\
& +Kn^{-\frac{p}{2}+1}E\Big( \int_{\kappa(n,t)}^t  \int_Z |\Lambda_2^{(n)}(x_{\kappa(n,s)}^n,z)|^p \nu(dz) ds \Big|\mathscr{F}_{\kappa(n,t)}\Big)\notag
\\
& +KE\Big( \int_{\kappa(n,t)}^t  \int_Z |\Lambda_2^{(n)}(x_{\kappa(n,s)}^n,z)|^p \nu(dz) ds \Big|\mathscr{F}_{\kappa(n,t)}\Big) \notag
\\
&+Kn^{-\frac{p}{2}+1} E\Big( \int_{\kappa(n,t)}^t  \int_Z |\Lambda_3^{(n)}(x_{\kappa(n,s)}^n,z)|^p \nu(dz)ds \Big|\mathscr{F}_{\kappa(n,t)}\Big)\notag
\\
& +K E\Big( \int_{\kappa(n,t)}^t  \int_Z |\Lambda_3^{(n)}(x_{\kappa(n,s)}^n,z)|^p \nu(dz)ds \Big|\mathscr{F}_{\kappa(n,t)}\Big) \notag
\\
&+Kn^{-p+1} E\Big( \int_{\kappa(n,t)}^t  \int_Z |\Lambda_3^{(n)}(x_{\kappa(n,s)}^n,z)|^p \nu(dz)ds\Big|\mathscr{F}_{\kappa(n,t)}\Big) \notag
\end{align}
and then Remark \ref{rem:growth(n)} completes the proof. 
\end{proof}
As a consequence of the above lemma, one obtains the following corollary. 
\begin{cor} \label{cor:sig(n):uc} 
Let Assumptions A-\ref{as:sde:ini} to A-\ref{as:sde:lip:der} be satisfied. Then, the following holds, 
\begin{align*}
E|\sigma^{(n)}(t,x_{\kappa(n,t)}^n)|^p\leq Kn^{p/8}E(1+|x_{\kappa(n,t)}^n|)^p
\end{align*}
for any $2\leq  p \leq p_0$, $t\in[0,T]$ and $n\in\mathbb{N}$ where the positive constant $K$ does not depend on $n$. 
\end{cor}
\begin{lem}  \label{lem:gam(n):c}
Let Assumptions A-\ref{as:sde:ini} to A-\ref{as:sde:lip:der} be satisfied. Then, the following holds, 
\begin{align*}
E\Big(\int_Z|\gamma^{(n)}(t,x_{\kappa(n,t)}^n,z)|^p\nu(dz)\Big|\mathscr{F}_{\kappa(n,t)}\Big)\leq K(1+|x_{\kappa(n,t)}^n|)^p
\end{align*}
almost surely for any $2\leq  p\leq p_0$, $t\in[0,T]$ and $n\in\mathbb{N}$ where the positive constant $K$ does not depend on $n$. 
\end{lem}
\begin{proof}
By using the definition in \eqref{eq:gamma(n)}, 
\begin{align*}
E\Big(\int_Z|\gamma^{(n)}(t,&x_{\kappa(n,t)}^n,z)|^p\nu(dz)\Big|\mathscr{F}_{\kappa(n,t)}\Big) \leq KE\Big(\int_Z|\gamma(x_{\kappa(n,t)}^n,z)|^p\nu(dz)\Big|\mathscr{F}_{\kappa(n,t)}\Big)
\\
&+KE\Big(\int_Z\Big|\int_{\kappa(n,t)}^t \sum_{k=1}^m \Gamma_1^{(n),k}(x_{\kappa(n,s)}^n,z) dw_s^k\Big|^p\nu(dz)\Big|\mathscr{F}_{\kappa(n,t)}\Big)
\\
&+KE\Big(\int_Z\Big|\int_{\kappa(n,t)}^t\int_Z\Gamma_2(x_{\kappa(n,s)}^n,z,z_1)\tilde{N}(ds,dz_1)\Big|^p\nu(dz)\Big|\mathscr{F}_{\kappa(n,t)}\Big)
\\
& +KE\Big(\int_Z\Big|\int_{\kappa(n,t)}^t\int_Z\Gamma_3(x_{\kappa(n,s)}^n,z,z_1)N(ds,dz_1)\Big|^p\nu(dz)\Big|\mathscr{F}_{\kappa(n,t)}\Big) 
\end{align*}
which on the application of Assumption A-\ref{as:sde:growth:gam}, an elementary inequality of stochastic integral and H\"{o}lder's inequality yields, 
\begin{align*}
E\Big(\int_Z|\gamma^{(n)}(t,&x_{\kappa(n,t)}^n,z)|^p\nu(dz)\Big|\mathscr{F}_{\kappa(n,t)}\Big) \leq K (1+|x_{\kappa(n,t)}^n|)^p
\\
&+Kn^{-\frac{p}{2}+1}E\Big(\int_Z\int_{\kappa(n,t)}^t \Big|\sum_{k=1}^m \Gamma_1^{(n),k}(x_{\kappa(n,s)}^n,z)\Big|^p ds \nu(dz)\Big|\mathscr{F}_{\kappa(n,t)}\Big)
\\
&+Kn^{-\frac{p}{2}+1}E\Big(\int_Z\int_{\kappa(n,t)}^t\int_Z|\Gamma_2(x_{\kappa(n,s)}^n,z,z_1)|^p\nu(dz_1)ds \nu(dz)\Big|\mathscr{F}_{\kappa(n,t)}\Big)
\\
&+KE\Big(\int_Z\int_{\kappa(n,t)}^t\int_Z|\Gamma_2(x_{\kappa(n,s)}^n,z,z_1)|^p\nu(dz_1)ds \nu(dz)\Big|\mathscr{F}_{\kappa(n,t)}\Big)
\\
&+Kn^{-\frac{p}{2}+1}E\Big(\int_Z\int_{\kappa(n,t)}^t\int_Z|\Gamma_3(x_{\kappa(n,s)}^n,z,z_1)|^p\nu(dz_1)ds\nu(dz)\Big|\mathscr{F}_{\kappa(n,t)}\Big) 
\\
& +KE\Big(\int_Z\int_{\kappa(n,t)}^t\int_Z|\Gamma_3(x_{\kappa(n,s)}^n,z,z_1)|^p\nu(dz_1)ds\nu(dz)\Big|\mathscr{F}_{\kappa(n,t)}\Big) 
\\
& +Kn^{-p+1}E\Big(\int_Z\int_{\kappa(n,t)}^t\int_Z|\Gamma_3(x_{\kappa(n,s)}^n,z,z_1)|^p\nu(dz_1)ds\nu(dz)\Big|\mathscr{F}_{\kappa(n,t)}\Big) 
\end{align*}
and then Remark \ref{rem:growth(n)} completes the proof. 
\end{proof}
As a consequence of the above lemma, one obtains the following corollary. 
\begin{cor} \label{cor:gam(n):uc} 
Let Assumptions A-\ref{as:sde:ini} to A-\ref{as:sde:lip:der} be satisfied. Then, the following holds, 
\begin{align*}
E\int_Z|\gamma^{(n)}(t,x_{\kappa(n,t)}^n,z)|^p\nu(dz)\leq KE(1+|x_{\kappa(n,t)}^n|)^p
\end{align*}
almost surely for any $2\leq  p \leq p_0$, $t\in[0,T]$ and $n\in\mathbb{N}$ where the positive constant $K$ does not depend on $n$. 
\end{cor}
\begin{lem} \label{lem:one-step:c:mb}
Let Assumptions A-\ref{as:sde:ini} to A-\ref{as:sde:lip:der} be satisfied. Then,  the following holds, 
\begin{align*}
E\big(|x_t^n-x_{\kappa(n,t)}^n|^{p}|\mathscr{F}_{\kappa(n,t)}\big)  & \leq K n^{-1}\big(1+|x_{\kappa(n,t)}^n|\big)^{p}, \mbox{ when } \, 8/3\leq p \leq p_0, \mbox{ and }
\\
E\big(|x_t^n-x_{\kappa(n,t)}^n|^{p}|\mathscr{F}_{\kappa(n,t)}\big)  & \leq K n^{-3p/8}\big(1+|x_{\kappa(n,t)}^n|\big)^{p}, \mbox{ when } \, 1\leq p < 8/3,
\end{align*}
almost surely for any  $t \in [0,T]$ and $n \in \mathbb{N}$ where positive constant $K$ does not depend on $n$.
\end{lem}
\begin{proof} By using the Milstein-type scheme defined in equation \eqref{eq:scheme}, one can write, 
\begin{align*}
E\big(|x_t^n-x_{\kappa(n,t)}^n|^{p}|\mathscr{F}_{\kappa(n,t)}\big) & \leq K E\Big(\Big|\int^{t}_{\kappa(n,t)} b^{(n)}(x_{\kappa(n,s)}^n)ds\Big|^{p}\Big|\mathscr{F}_{\kappa(n,t)}\Big) \notag
\\
&+K E\Big(\Big|\int^{t}_{\kappa(n,t)} \sigma^{(n)}(s, x_{\kappa(n,s)}^n)dw_s\Big|^{p}\Big|\mathscr{F}_{\kappa(n,t)}\Big)
\\
& + K E\Big(\Big|\int^{t}_{\kappa(n,t)}\int_Z \gamma^{(n)}(s, x_{\kappa(n,s)}^n, z)\tilde N(ds,dz)\Big|^{p}\Big|\mathscr{F}_{\kappa(n,t)}\Big)
\end{align*}
which on the application of an elementary inequality of stochastic integral and H\"older's inequality yields, 
\begin{align}
E\big(|x_t^n-x_{\kappa(n,t)}^n|^{p}|\mathscr{F}_{\kappa(n,t)}\big) & \leq K n^{-p+1}E \Big(\int_{\kappa(n,t)}^{t}|b^{(n)}(x_{\kappa(n,s)}^n)|^{p} ds \Big|\mathscr{F}_{\kappa(n,t)}\Big) \notag 
\\
&+ K n^{-\frac{p}{2}+1}E\Big(\int_{\kappa(n,t)}^{t}|\sigma^{(n)}(s,x_{\kappa(n,s)}^n)|^{p}ds\Big|\mathscr{F}_{\kappa(n,t)}\Big) \notag
\\
& + K n^{-\frac{p}{2}+1}E\Big(\int^{t}_{\kappa(n,t)}\int_Z |\gamma^{(n)}(s, x_{\kappa(n,s)}^n, z)|^p  \nu(dz) ds\Big|\mathscr{F}_{\kappa(n,t)}\Big) \notag
\\
& +K E\Big( \int^{t}_{\kappa(n,t)}\int_Z |\gamma^{(n)}(s,x_{\kappa(n,s)}^n, z)|^p  \nu(dz)ds\Big|\mathscr{F}_{\kappa(n,t)}\Big) \notag
\end{align}
and then the application of Lemmas [\ref{lem:sig(n):c}, \ref{lem:gam(n):c}] completes the proof. 
\end{proof}
As a consequence of the above lemma, one obtains the following corollary. 
\begin{cor} \label{cor:one-step:uc:mb}
Let Assumptions A-\ref{as:sde:ini} to A-\ref{as:sde:lip:der} be satisfied. Then,  the following holds, 
\begin{align*}
E|x_t^n-x_{\kappa(n,t)}^n|^{p} & \leq K n^{-1}E\big(1+|x_{\kappa(n,t)}^n|\big)^{p}, \mbox{ when } \, 8/3\leq p \leq p_0, \mbox{ and }
\\
E|x_t^n-x_{\kappa(n,t)}^n|^{p}|  & \leq K n^{-3p/8}E\big(1+|x_{\kappa(n,t)}^n|\big)^{p}, \mbox{ when } \, 1\leq p < 8/3, 
\end{align*}
for any  $t \in [0,T]$ and $n \in \mathbb{N}$ where positive constant $K$ does not depend on $n$.
\end{cor}
The following lemma establishes the moment bounds of the Milstein-type scheme defined in equation \eqref{eq:scheme}. 
\begin{lem} \label{lem:mb:scheme}
Let Assumptions A-\ref{as:sde:ini} to A-\ref{as:sde:lip:der} be satisfied. Then, the following holds, 
$$
\sup_{n \in \mathbb{N}}\sup_{0 \leq t \leq T}E|x_t^n|^{p_0} \leq K,
$$
where $K:=K(L,T,p_0,m,d,E|\xi|^{p_0})$ is a positive constant which does not depend on $n$.
\end{lem}
\begin{proof} 
Let us assume that  $p_0 \geq 4$. By using It\^{o}'s formula for the Milstein-type scheme defined in equation \eqref{eq:scheme},
\begin{align}
|x_t^n|^{p_0} = &|\xi|^{p_0}+ p_0 \int_0^{t} |x_s^n|^{p_0-2} x_s^n b^{(n)}(x_{\kappa(n,s)}^n) ds + p_0\int_{0}^{t} |x_s^n|^{p_0-2} x_s^n \sigma^{(n)}(s,x_{\kappa(n,s)}^n)  dw_s  \notag
\\
& + \frac{p_0(p_0-2)}{2} \int_{0}^{t} |x_s^n|^{p_0-4}| \sigma^{(n)*}(s,x_{\kappa(n,s)}^n)x_s^n|^2 ds  + \frac{p_0}{2} \int_{t_0}^{t} |x_s^n|^{p_0-2}|\sigma^{(n)}(s,x_{\kappa(n,s)}^n)|^2 ds  \notag
\\
&+  p_0 \int_{0}^{t} \int_{Z} |x_s^n|^{p_0-2} x_{s}^n  \gamma^{(n)}(s,x_{\kappa(n,s)}^n,z)   \tilde N(ds,dz) \notag
\\
+\int_{0}^{t} & \int_{Z} \big\{ |x_{s}^n+ \gamma^{(n)}(s,x_{\kappa(n,s)}^n,z)|^{p_0}-|x_{s}^n|^{p_0}-p_0|x_{s}^n|^{p_0-2} x_{s}^n  \gamma^{(n)}(s,x_{\kappa(n,s)}^n,z) \big\}N(ds,dz) \notag
\end{align}
almost surely for any $t \in [0,T]$ and $n \in \mathbb{N}$. By taking expectation on both sides, martingale terms vanish hence one obtains,  
\begin{align}
E|x_t^n|^{p_0} = &E|\xi|^{p_0}+ p_0 E \int_0^{t} |x_s^n|^{p_0-2} x_s^n b^{(n)}(x_{\kappa(n,s)}^n) ds \notag 
\\
& + \frac{p_0(p_0-1)}{2} E\int_{0}^{t} |x_s^n|^{p_0-2}| \sigma^{(n)}(s,x_{\kappa(n,s)}^n)|^2 ds  \notag
\\
+E\int_{0}^{t} \int_{Z}& \big\{ |x_{s}^n+ \gamma^{(n)}(s,x_{\kappa(n,s)}^n,z)|^{p_0}-|x_{s}^n|^{p_0}-p_0|x_{s}^n|^{p_0-2} x_{s}^n  \gamma^{(n)}(s,x_{\kappa(n,s)}^n,z) \big\}\nu(dz)ds \notag
\end{align}
for any $t\in[0,T]$ and $n\in \mathbb{N}$. 
For last term on the right hand side of the above equation, one applies formula for the remainder for map $y \in \mathbb{R}^d \to |y|^{p_0}$, 
\begin{align*}
|y_1+y_2|^{p_0}-|y_1|^{p_0}-p_0 |y_1|^{p_0-2}y_1y_2 \leq K\int_0^1|y_1+sy_2|^{p_0-2}|y_2|^2 ds \leq K(|y_1|^{p_0-2}|y_2|^2+|y_2|^{p_0})
\end{align*}  
for any $y_1,y_2\in \mathbb{R}^d$ which therefore gives,  
\begin{align}
E|x_t^n|^{p_0} \leq  &E|\xi|^{p_0}+ p_0 E \int_0^{t} |x_s^n|^{p_0-2} x_s^n b^{(n)}(x_{\kappa(n,s)}^n) ds \notag 
\\
& + \frac{p_0(p_0-1)}{2} E\int_{0}^{t} |x_s^n|^{p_0-2}| \sigma^{(n)}(s,x_{\kappa(n,s)}^n)|^2 ds  \notag
\\
&+K E\int_{0}^{t} \int_{Z}|x_{s}^n|^{p_0-2} |\gamma^{(n)}(s,x_{\kappa(n,s)}^n,z)|^{2}\nu(dz)ds \notag
\\
&+K E\int_{0}^{t} \int_{Z}  |\gamma^{(n)}(s,x_{\kappa(n,s)}^n,z)|^{p_0}\nu(dz)ds. \notag
\end{align}
for any $t\in[0,T]$ and $n\in\mathbb{N}$. Now, one uses the following inequality for estimating the third term on the right hand side of the above inequality,  
\begin{align*}
|y_0&+y_1+y_2+y_3|^2=|y_0|^2+|y_1|^2+|y_2|^2+|y_3|^2+2\sum_{i=1}^d\sum_{j=1}^my_0^{ij}y_1^{ij}+2\sum_{i=1}^d\sum_{j=1}^my_0^{ij}y_2^{ij}
\\
&\qquad+2\sum_{i=1}^d\sum_{j=1}^my_0^{ij}y_3^{ij}+2\sum_{i=1}^d\sum_{j=1}^my_1^{ij}y_2^{ij}+2\sum_{i=1}^d\sum_{j=1}^m y_1^{ij}y_3^{ij}+2\sum_{i=1}^d\sum_{j=1}^my_2^{ij}y_3^{ij}
\\
&\leq |y_0|^2+K|y_1|^2+K|y_2|^2+K|y_3|^2+2\sum_{i=1}^d\sum_{j=1}^my_0^{ij}y_1^{ij}+2\sum_{i=1}^d\sum_{j=1}^my_0^{ij}y_2^{ij}+2\sum_{i=1}^d\sum_{j=1}^my_0^{ij}y_3^{ij}
\end{align*}
for any $y_0, y_1, y_2, y_3 \in \mathbb{R}^{d \times m}$ to obtain the following,
\begin{align}
E|x_t^n|^{p_0} \leq  &E|\xi|^{p_0}+ p_0 E \int_0^{t} |x_s^n|^{p_0-2} (x_s^n-x_{\kappa(n,s)}^n) b^{(n)}(x_{\kappa(n,s)}^n) ds \notag 
\\
& + \frac{p_0}{2} E\int_{0}^{t} |x_s^n|^{p_0-2}\{2x_{\kappa(n,s)}^nb^{(n)}(x_{\kappa(n,s)}^n)+(p_0-1)| \Lambda_0^{(n)}(x_{\kappa(n,s)}^n)|^2\} ds  \notag
\\
& + K E\int_{0}^{t} |x_s^n|^{p_0-2}\Big|\int_{\kappa(n,s)}^s \sum_{k=1}^{m}
\Lambda_1^{(n),k}(x_{\kappa(n,r)})dw_r^k\Big|^2ds \notag
\\
& + K E\int_{0}^{t} |x_s^n|^{p_0-2}\Big|\int_{\kappa(n,s)}^s\int_Z\Lambda_2^{(n)}(x_{\kappa(n,r)}^n,z) \tilde{N}(dr,dz)\Big|^2 ds \notag
\\
& + K E\int_{0}^{t} |x_s^n|^{p_0-2}\Big|\int_{\kappa(n,s)}^s\int_Z\Lambda_3^{(n)}(x_{\kappa(n,r)}^n,z) N(dr,dz)\Big|^2 ds \notag
\\
& + KE\int_{0}^{t} |x_s^n|^{p_0-2}\sum_{i=1}^d\sum_{j=1}^m\Lambda_0^{(n),ij}(x_{\kappa(n,s)}^n)\int_{\kappa(n,s)}^s\sum_{k=1}^m\Lambda_1^{(n),k,ij}(x_{\kappa(n,r)}^n)dw_r^k ds \notag
\\
& + KE\int_{0}^{t} |x_s^n|^{p_0-2}\sum_{i=1}^d\sum_{j=1}^m\Lambda_0^{(n),ij}(x_{\kappa(n,s)}^n)\int_{\kappa(n,s)}^s\int_Z\Lambda_2^{(n),ij}(x_{\kappa(n,r)}^n,z)\tilde{N}(dr,dz) ds \notag
\\
& + KE\int_{0}^{t} |x_s^n|^{p_0-2}\sum_{i=1}^d\sum_{j=1}^m\Lambda_0^{(n),ij}(x_{\kappa(n,s)}^n)\int_{\kappa(n,s)}^s\int_Z\Lambda_3^{(n),ij}(x_{\kappa(n,r)}^n,z)N(dr,dz) ds \notag
\\
&+K E\int_{0}^{t} |x_{s}^n|^{p_0} ds +K E\int_{0}^{t} \int_{Z}  |\gamma(x_{\kappa(n,s)}^n,z)|^{p_0}\nu(dz)ds \notag
\\
&+K E\int_{0}^{t} \int_{Z}  \Big|\int_{\kappa(n,s)}^s\sum_{k=1}^m\Gamma_1^{(n),k}(x_{\kappa(n,r)}^n,z)dw_r^k\Big|^{p_0}\nu(dz)ds \notag
\\
&+K E\int_{0}^{t} \int_{Z}  \Big|\int_{\kappa(n,s)}^s\int_Z\Gamma_2(x_{\kappa(n,r)}^n,z,z_1)\tilde{N}(dr,dz_1)\Big|^{p_0}\nu(dz)ds \notag
\\
&+K E\int_{0}^{t} \int_{Z}  \Big|\int_{\kappa(n,s)}^s\int_Z\Gamma_3(x_{\kappa(n,r)}^n,z,z_1)N(dr,dz_1)\Big|^{p_0}\nu(dz)ds \notag
\\
&=:E|\xi|^{p_0}+K+K\int_0^t\sup_{0 \leq r \leq s}E|x_r^n|^{p_0}ds \notag
\\
& +F_1+F_2+F_3+F_4+F_5+F_6+F_7+F_8+F_8+F_9+F_{10}+F_{11} \label{eq:F1+F11}
\end{align}
for any $t \in [0,T]$ and $n\in\mathbb{N}$. For estimating $F_1$, one proceeds as follows, 
\begin{align}
F_1&:=p_0 E \int_0^{t} |x_s^n|^{p_0-2} (x_s^n-x_{\kappa(n,s)}^n) b^{(n)}(x_{\kappa(n,s)}^n) ds \notag
\\
&\leq K E \int_0^{t} |x_s^n-x_{\kappa(n,s)}^n|^{p_0-1} |b^{(n)}(x_{\kappa(n,s)}^n)| ds \notag
\\
&+ K E \int_0^{t} |x_{\kappa(n,s)}^n|^{p_0-2}|x_s^n-x_{\kappa(n,s)}^n| |b^{(n)}(x_{\kappa(n,s)}^n)| ds \notag
\end{align}
which on using Remark \ref{rem:growth(n)} yields,  
\begin{align*}
F_1&\leq K n^{1/4}  E \int_0^{t} (1+|x_{\kappa(n,s)}^n|) E(|x_s^n-x_{\kappa(n,s)}|^{p_0-1}|\mathscr{F}_{\kappa(n,s)})  ds \notag
\\
&+n^{1/4} K E \int_0^{t} (1+|x_{\kappa(n,s)}^n|^{p_0-1}) E(|x_s^n-x_{\kappa(n,s)}^n| | \mathscr{F}_{\kappa(n,s)}) ds \notag 
\end{align*}
and then the application of Lemma \ref{lem:one-step:c:mb} gives the following estimates, 
\begin{align}
F_1&\leq K n^{-\frac{1}{2}}  E \int_0^{t} (1+|x_{\kappa(n,s)}^n|^{p_0})  ds \leq K+K\int_0^t \sup_{0 \leq r \leq s} E|x_r^n|^{p_0}ds \label{eq:F1} 
\end{align}
for any $t \in [0,T]$ and $n\in\mathbb{N}$. Moreover, $F_2$ can be estimated as follows, 
\begin{align}
F_2 &:= \frac{p_0}{2} E\int_{0}^{t} |x_s^n|^{p_0-2}\{2x_{\kappa(n,s)}^nb^{(n)}(x_{\kappa(n,s)}^n)+(p_0-1)| \Lambda_0^{(n)}(x_{\kappa(n,s)}^n)|^2\} ds \notag
\\
&\leq \frac{p_0}{2} E\int_{0}^{t} |x_s^n|^{p_0-2}\frac{2x_{\kappa(n,s)}^nb(x_{\kappa(n,s)}^n)+(p_0-1)| \sigma(x_{\kappa(n,s)}^n)|^2}{1+n^{-1}|x_{\kappa(n,s)}^n|^{4 \rho}} ds \notag
\\
&\leq K E\int_{0}^{t} |x_s^n|^{p_0-2}(1+|x_{\kappa(n,s)}^n|^2) ds \leq K+K\int_0^t \sup_{0 \leq r \leq s} E|x_r^n|^{p_0}ds \label{eq:F2}
\end{align}
for any $t \in [0,T]$ and $n\in\mathbb{N}$. Furthermore, one uses Young's inequality and an elementary inequality of stochastic integral to obtain, 
\begin{align}
F_3 & := K E\int_{0}^{t} |x_s^n|^{p_0-2}\Big|\int_{\kappa(n,s)}^s \sum_{k=1}^{m}
\Lambda_1^{(n),k}(x_{\kappa(n,r)}^n)dw_r^k\Big|^2ds \notag
\\
&\leq K E\int_{0}^{t} |x_s^n|^{p_0} ds + K E\int_{0}^{t} \Big|\int_{\kappa(n,s)}^s \sum_{k=1}^{m}
\Lambda_1^{(n),k}(x_{\kappa(n,r)}^n)dw_r^k\Big|^{p_0}ds \notag
\\
&\leq K E\int_{0}^{t} |x_s^n|^{p_0} ds + K n^{-\frac{p_0}{2}+1}E\int_{0}^{t} \int_{\kappa(n,s)}^s \Big|\sum_{k=1}^{m}
\Lambda_1^{(n),k}(x_{\kappa(n,r)}^n)\Big|^{p_0} dr ds \notag
\end{align}
and then the application of Remark \ref{rem:growth(n)} implies, 
\begin{align}
F_3 \leq K+ K \int_{0}^{t} \sup_{0 \leq r \leq s}E|x_r^n|^{p_0} ds \label{eq:F3}
\end{align}
for any $t\in [0,T]$ and $n\in\mathbb{N}$. For $F_4$, one writes 
\begin{align}
F_4&:= K E\int_{0}^{t} |x_s^n|^{p_0-2}\Big|\int_{\kappa(n,s)}^s\int_Z\Lambda_2^{(n)}(x_{\kappa(n,r)}^n,z) \tilde{N}(dr,dz)\Big|^2 ds \notag
\\
&\leq K E\int_{0}^{t} n^{\frac{p_0-2}{p_0}} |x_s^n-x_{\kappa(n,s)}^n|^{p_0-2} n^{-\frac{p_0-2}{p_0}}\Big|\int_{\kappa(n,s)}^s\int_Z\Lambda_2^{(n)}(x_{\kappa(n,r)}^n,z) \tilde{N}(dr,dz)\Big|^2 ds \notag
\\
&+K E\int_{0}^{t} |x_{\kappa(n,s)}^n|^{p_0-2}\Big|\int_{\kappa(n,s)}^s\int_Z\Lambda_2^{(n)}(x_{\kappa(n,r)}^n,z) \tilde{N}(dr,dz)\Big|^2  ds \notag
\end{align}
and then due to Young's inequality and an elementary inequality of stochastic integral, it follows that, 
\begin{align}
F_4&\leq  K n \int_{0}^{t}  E|x_s^n-x_{\kappa(n,s)}^n|^{p_0} ds \notag 
\\
&+  K n^{-\frac{p_0-2}{2}-\frac{p_0}{2}+1}E \int_{0}^{t}\int_{\kappa(n,s)}^s\int_Z|\Lambda_2^{(n)}(x_{\kappa(n,r)}^n,z)|^{p_0} \nu(dz)dr ds \notag
\\
& +  K n^{-\frac{p_0-2}{2}}E \int_{0}^{t}\int_{\kappa(n,s)}^s\int_Z|\Lambda_2^{(n)}(x_{\kappa(n,r)}^n,z)|^{p_0} \nu(dz) dr ds \notag
\\
&+K E\int_{0}^{t} |x_{\kappa(n,s)}^n|^{p_0-2}\int_{\kappa(n,s)}^s\int_Z|\Lambda_2^{(n)}(x_{\kappa(n,r)}^n,z)|^2 \nu(dz)dr ds \notag
\end{align}
 which on the application of Remark \ref{rem:growth(n)} and Corollary  \ref{cor:one-step:uc:mb} yields the following estimates, 
 \begin{align}
 F_4 \leq K + K \int_0^t \sup_{0 \leq r \leq s} E|x_r^n|^{p_0} ds  \label{eq:F4} 
 \end{align}
 for any $t \in [0,T]$ and $n\in\mathbb{N}$. For estimating $F_5$, one proceeds as follows,  
\begin{align}
F_5&:=K E\int_{0}^{t} |x_s^n|^{p_0-2}\Big|\int_{\kappa(n,s)}^s\int_Z\Lambda_3^{(n)}(x_{\kappa(n,r)}^n,z) N(dr,dz)\Big|^2 ds \notag
\\
&\leq K E\int_{0}^{t} |x_s^n|^{p_0-2}\Big|\int_{\kappa(n,s)}^s\int_Z\Lambda_3^{(n)}(x_{\kappa(n,r)}^n,z)\tilde{N}(dr,dz)\Big|^2 ds \notag
\\
&+K E\int_{0}^{t} |x_s^n|^{p_0-2}\Big|\int_{\kappa(n,s)}^s\int_Z\Lambda_3^{(n)}(x_{\kappa(n,r)}^n,z)\nu(dz)dr\Big|^2 ds \notag
\\
&\leq K E\int_{0}^{t} n^\frac{p_0-2}{p_0}|x_s^n-x_{\kappa(n,s)}^n|^{p_0-2} n^{-\frac{p_0-2}{p_0}}\Big|\int_{\kappa(n,s)}^s\int_Z\Lambda_3^{(n)}(x_{\kappa(n,r)}^n,z) \tilde{N}(dr,dz)\Big|^2 ds \notag
\\
& + K E\int_{0}^{t} |x_{\kappa(n,s)}^n|^{p_0-2}\Big|\int_{\kappa(n,s)}^s\int_Z\Lambda_3^{(n)}(x_{\kappa(n,r)}^n,z) \tilde{N}(dr,dz)\Big|^2 ds \notag
\\
&+Kn^{-1} E\int_{0}^{t} |x_s^n|^{p_0-2} \int_{\kappa(n,s)}^s\int_Z|\Lambda_3^{(n)}(x_{\kappa(n,r)}^n,z)|^2 \nu(dz)dr ds \notag
\end{align} 
which on the application of Young's inequality and an elementary inequality of stochastic integrals give, 
\begin{align}
F_5 & \leq K n E\int_{0}^{t} |x_s^n-x_{\kappa(n,s)}^n|^{p_0} ds + K n^{-\frac{p_0-2}{2}-\frac{p_0}{2}+1} E\int_{0}^{t} \int_{\kappa(n,s)}^s\int_Z|\Lambda_3^{(n)}(x_{\kappa(n,r)}^n,z)|^{p_0} \nu(dz)dr ds \notag
\\
&+ K n^{-\frac{p_0-2}{2}} E\int_{0}^{t} \int_{\kappa(n,s)}^s\int_Z|\Lambda_3^{(n)}(x_{\kappa(n,r)}^n,z)|^{p_0} \nu(dz)dr  ds  \notag
\\
&+ K E\int_{0}^{t} |x_{\kappa(n,s)}^n|^{p_0-2}\int_{\kappa(n,s)}^s\int_Z|\Lambda_3^{(n)}(x_{\kappa(n,r)}^n,z)|^2 \nu(dz) dr  ds \notag
\\
&+Kn^{-1} E\int_{0}^{t} |x_s^n|^{p_0-2} \int_{\kappa(n,s)}^s\int_Z|\Lambda_3^{(n)}(x_{\kappa(n,r)}^n,z)|^2 \nu(dz)dr ds \notag
\end{align}
and then due to Remark \ref{rem:growth(n)}, Corollary \ref{cor:one-step:uc:mb} and Young's inequality, one obtains, 
\begin{align}
F_5\leq K + K \int_0^t \sup_{0\leq r \leq s}E|x_r^n|^{p_0} ds \label{eq:F5}
\end{align}
for any $t\in[0,T]$ and $n\in\mathbb{N}$.  In order to estimate $F_6$, one uses Remark \ref{rem:growth(n)} to obtain the following,  
 \begin{align}
 F_6& := KE\int_{0}^{t} |x_s^n|^{p_0-2}\sum_{i=1}^d\sum_{j=1}^m\Lambda_0^{(n),ij}(x_{\kappa(n,s)}^n)\int_{\kappa(n,s)}^s\sum_{k=1}^m\Lambda_1^{(n),k,ij}(x_{\kappa(n,r)}^n)dw_r^k ds \notag
  \\
 & \leq KE\int_{0}^{t} |x_s^n|^{p_0-2}n^{1/8}(1+|x_{\kappa(n,s)}^n|)\Big|\int_{\kappa(n,s)}^s\sum_{k=1}^m\Lambda_1^{(n),k}(x_{\kappa(n,r)}^n)dw_r^k\Big| ds \notag
 \\
 & \leq KE\int_{0}^{t} |x_s^n-x_{\kappa(n,s)}^n|^{p_0-2}n^{\frac{p_0-2}{p_0}}(1+|x_{\kappa(n,s)}^n|)n^{\frac{1}{8}-\frac{p_0-2}{p_0}}\Big|\int_{\kappa(n,s)}^s\sum_{k=1}^m\Lambda_1^{(n),k}(x_{\kappa(n,r)}^n)dw_r^k\Big| ds \notag
 \\
&+ KE\int_{0}^{t} (1+|x_{\kappa(n,s)}^n|^{p_0-1})n^{\frac{1}{8}}\Big|\int_{\kappa(n,s)}^s\sum_{k=1}^m\Lambda_1^{(n),k}(x_{\kappa(n,r)}^n)dw_r^k\Big| ds \notag 
 \end{align}
 which on the application of Young's inequality and an elementary inequality of stochastic integral gives the following estimates, 
 \begin{align}
 F_6 & \leq K n E\int_{0}^{t} |x_s^n-x_{\kappa(n,s)}^n|^{p_0}ds \notag
 \\
 & +K n^{\frac{p_0}{16}-\frac{p_0-2}{2}-\frac{p_0}{4}+1}E\int_0^t (1+|x_{\kappa(n,s)}^n|)^\frac{p_0}{2}\int_{\kappa(n,s)}^s\Big|\sum_{k=1}^m\Lambda_1^{(n),k}(x_{\kappa(n,r)}^n)\Big|^\frac{p_0}{2} dr ds \notag
 \\
&+ KE\int_{0}^{t} (1+|x_{\kappa(n,s)}^n|^{p_0-1})n^{\frac{1}{8}}\Big(\int_{\kappa(n,s)}^s\Big|\sum_{k=1}^m\Lambda_1^{(n),k}(x_{\kappa(n,r)}^n)\Big|^2 dr\Big)^\frac{1}{2} ds \notag
 \end{align}
 and then due to Remark \ref{rem:growth(n)} and Corollary \ref{cor:one-step:uc:mb}, one obtains, 
 \begin{align}
 F_6\leq K+ K\int_0^t \sup_{0 \leq r \leq s}E|x_r^n|^{p_0}ds \label{eq:F6}
 \end{align}
 for any $t \in [0,T]$ and $n\in\mathbb{N}$. Similarly, for estimating $F_7$, one writes, 
 \begin{align}
 F_7&:=KE\int_{0}^{t} |x_s^n|^{p_0-2}\sum_{i=1}^d\sum_{j=1}^m\Lambda_0^{(n),ij}(x_{\kappa(n,s)}^n)\int_{\kappa(n,s)}^s\int_Z\Lambda_2^{(n),ij}(x_{\kappa(n,r)}^n,z)\tilde{N}(dr,dz) ds \notag
 \\
 \leq & KE\int_{0}^{t} n^\frac{p_0-2}{p_0}|x_s^n-x_{\kappa(n,s)}^n|^{p_0-2}n^{\frac{1}{8}-\frac{p_0-2}{p_0}}(1+|x_{\kappa(n,s)}^n|)\Big|\int_{\kappa(n,s)}^s\int_Z\Lambda^{(n)}_2(x_{\kappa(n,r)}^n,z)\tilde{N}(dr,dz)\Big| ds \notag 
 \\
 &+KE\int_{0}^{t} (1+|x_{\kappa(n,s)}^n|^{p_0-1})n^\frac{1}{8}\Big|\int_{\kappa(n,s)}^s\int_Z\Lambda^{(n)}_2(x_{\kappa(n,r)}^n,z)\tilde{N}(dr,dz)\Big| ds \notag  
 \end{align}
 and then one uses Young's inequality and an elementary inequality of stochastic integral to obtain, 
\begin{align}
F_7  &\leq K n E\int_{0}^{t} |x_s^n-x_{\kappa(n,s)}^n|^{p_0} ds \notag
\\
& + K n^{\frac{p_0}{16}-\frac{p_0-2}{2}-\frac{p_0}{4}+1}E\int_{0}^{t}(1+|x_{\kappa(n,s)}^n|)^\frac{p_0}{2}\int_{\kappa(n,s)}^s\int_Z|\Lambda_2^{(n)}(x_{\kappa(n,r)}^n,z)|^\frac{p_0}{2}\nu(dz)dr ds \notag 
\\
& + K n^{\frac{p_0}{16}-\frac{p_0-2}{2}}E\int_{0}^{t}(1+|x_{\kappa(n,s)}^n|)^\frac{p_0}{2}\int_{\kappa(n,s)}^s\int_Z|\Lambda_2^{(n)}(x_{\kappa(n,r)}^n,z)|^\frac{p_0}{2}\nu(dz)dr ds \notag 
 \\
 &+KE\int_{0}^{t} (1+|x_{\kappa(n,s)}^n|^{p_0-1})n^\frac{1}{8}\Big(\int_{\kappa(n,s)}^s\int_Z|\Lambda_2^{(n)}(x_{\kappa(n,r)}^n,z)|^2\nu(dz) dr\Big)^\frac{1}{2} ds \notag  
\end{align} 
which due to Remark \ref{rem:growth(n)} and Corollary \ref{cor:one-step:uc:mb} yields, 
\begin{align}
F_7 \leq K+K\int_0^t \sup_{0\leq r \leq s} E|x_r^n| ds \label{eq:F7} 
\end{align}
for any $t \in [0,T]$ and $n\in\mathbb{N}$. Further, due to Remark \ref{rem:growth(n)}, one has
\begin{align}
F_8& :=KE\int_{0}^{t} |x_s^n|^{p_0-2}\sum_{i=1}^d\sum_{j=1}^m\Lambda_0^{(n),ij}(x_{\kappa(n,s)}^n)\int_{\kappa(n,s)}^s\int_Z\Lambda_3^{(n),ij}(x_{\kappa(n,r)}^n,z)N(dr,dz) ds \notag
\\
&=KE\int_{0}^{t} |x_s^n|^{p_0-2}\sum_{i=1}^d\sum_{j=1}^m\Lambda_0^{(n),ij}(x_{\kappa(n,s)}^n)\int_{\kappa(n,s)}^s\int_Z\Lambda_3^{(n),ij}(x_{\kappa(n,r)}^n,z)\tilde{N}(dr,dz) ds \notag
\\
&+KE\int_{0}^{t} |x_s^n|^{p_0-2}\sum_{i=1}^d\sum_{j=1}^m\Lambda_0^{(n),ij}(x_{\kappa(n,s)}^n)\int_{\kappa(n,s)}^s\int_Z\Lambda_3^{(n),ij}(x_{\kappa(n,r)}^n,z)\nu(dz)dr ds \notag
\\
& \leq  KE\int_{0}^{t} n^\frac{p_0-2}{p_0}|x_s^n-x_{\kappa(n,s)}^n|^{p_0-2} n^{\frac{1}{8}-\frac{p_0-2}{p_0}}(1+|x_{\kappa(n,s)}^n|)\Big|\int_{\kappa(n,s)}^s\int_Z\Lambda_3^{(n)}(x_{\kappa(n,r)}^n,z)\tilde{N}(dr,dz)\Big| ds \notag
\\
& + KE\int_{0}^{t} |x_{\kappa(n,s)}^n|^{p_0-2} n^\frac{1}{8}(1+|x_{\kappa(n,s)}^n|)\Big|\int_{\kappa(n,s)}^s\int_Z\Lambda_3^{(n)}(x_{\kappa(n,r)}^n,z)\tilde{N}(dr,dz)\Big| ds \notag
\\
&+KE\int_{0}^{t} |x_s^n|^{p_0-2}n^\frac{1}{8}(1+|x_{\kappa(n,s)}^n|)\int_{\kappa(n,s)}^s\int_Z|\Lambda_3^{(n)}(x_{\kappa(n,r)}^n,z)|\nu(dz)dr ds \notag
\end{align} 
which on applying Young's inequality and an elementary inequality of stochastic integrals gives, 
\begin{align}
F_8 & \leq Kn E\int_{0}^{t} |x_s^n-x_{\kappa(n,s)}^n|^{p_0} ds \notag
\\
&+  K n^{\frac{p_0}{16}-\frac{p_0-2}{2}-\frac{p_0}{4}+1} E\int_{0}^{t} (1+|x_{\kappa(n,s)}^n|)^\frac{p_0}{2}\int_{\kappa(n,s)}^s\int_Z\Lambda_3^{(n)}(x_{\kappa(n,r)}^n,z)|^\frac{p_0}{2}\nu(dz) dr ds \notag
\\
&+  K n^{\frac{p_0}{16}-\frac{p_0-2}{2}} E\int_{0}^{t} (1+|x_{\kappa(n,s)}^n|)^\frac{p_0}{2}\int_{\kappa(n,s)}^s\int_Z\Lambda_3^{(n)}(x_{\kappa(n,r)}^n,z)|^\frac{p_0}{2}\nu(dz) dr ds \notag
\\
& + KE\int_{0}^{t} |x_{\kappa(n,s)}^n|^{p_0-2} n^\frac{1}{8}(1+|x_{\kappa(n,s)}^n|)\Big(\int_{\kappa(n,s)}^s\int_Z|\Lambda_3^{(n)}(x_{\kappa(n,r)}^n,z)|^2 \nu(dz) dr\Big)^\frac{1}{2} ds \notag 
\\
&+KE\int_{0}^{t} |x_s^n|^{p_0-2}n^\frac{1}{8}(1+|x_{\kappa(n,s)}^n|) \int_{\kappa(n,s)}^s\int_Z|\Lambda_3^{(n)}(x_{\kappa(n,r)}^n,z)|\nu(dz)dr ds \notag
\end{align}
and then the application of Remark \ref{rem:growth(n)} and Corollary \ref{cor:one-step:uc:mb} implies 
\begin{align}
F_8 & \leq K+K\int_0^t \sup_{0 \leq r \leq s}E|x_r^n|^{p_0} ds \label{eq:F8} 
\end{align} 
for any $t\in [0,T]$ and $n\in\mathbb{N}$. For estimating $F_9$, one uses an elementary inequality of stochastic integral and obtains the following,  
\begin{align}
F_9&:=K E\int_{0}^{t} \int_{Z}  \Big|\int_{\kappa(n,s)}^s\sum_{k=1}^m\Gamma_1^{(n),k}(x_{\kappa(n,r)}^n,z)dw_r^j\Big|^{p_0}\nu(dz)ds \notag
\\
&\leq K n^{-\frac{p_0}{2}+1}E\int_{0}^{t} \int_{Z}  \int_{\kappa(n,s)}^s|\Gamma_1^{(n)}(x_{\kappa(n,r)}^n,z)|^{p_0}dr\nu(dz)ds \notag
\end{align}
which on the application of Remark \ref{rem:growth(n)} gives, 
\begin{align}
F_9\leq K+K\int_0^t \sup_{0 \leq r \leq s}E|x_r^n|^{p_0} ds \label{eq:F9} 
\end{align}
for any $t \in [0,T]$ and $n\in\mathbb{N}$. Again, due to an elementary inequality of stochastic integral, one obtains the following estimates, 
\begin{align}
F_{10}& :=K E\int_{0}^{t} \int_{Z}  \Big|\int_{\kappa(n,s)}^s\int_Z\Gamma_2(x_{\kappa(n,r)}^n,z,z_1)\tilde{N}(dr,dz_1)\Big|^{p_0}\nu(dz)ds \notag 
\\
&\leq K n^{-\frac{p_0}{2}+1} E\int_{0}^{t} \int_{Z}  \int_{\kappa(n,s)}^s\int_Z |\Gamma_2(x_{\kappa(n,r)}^n,z,z_1)|^{p_0}dr \nu(dz_1)\nu(dz)ds \notag
\\
& +K  E\int_{0}^{t} \int_{Z}  \int_{\kappa(n,s)}^s\int_Z |\Gamma_2(x_{\kappa(n,r)}^n,z,z_1)|^{p_0}dr \nu(dz_1)\nu(dz)ds \notag
\end{align} 
and then Remark \ref{rem:growth(n)} gives the following estimates,
\begin{align}
F_{10}\leq K+K\int_0^t \sup_{0 \leq r \leq s} E|x_r^n|^{p_0} ds \label{eq:F10}
\end{align} 
for any $t \in [0,T]$ and $n\in\mathbb{N}$. For estimating $F_{11}$, one uses an elementary inequality of stochastic integrals to get, 
\begin{align}
F_{11}&:=K E\int_{0}^{t} \int_{Z}  \Big|\int_{\kappa(n,s)}^s\int_Z\Gamma_3(x_{\kappa(n,r)}^n,z,z_1)N(dr,dz_1)\Big|^{p_0}\nu(dz)ds \notag
\\
&\leq K E\int_{0}^{t} \int_{Z}  \Big|\int_{\kappa(n,s)}^s\int_Z\Gamma_3(x_{\kappa(n,r)}^n,z,z_1)\tilde{N}(dr,dz_1)\Big|^{p_0}\nu(dz)ds \notag
\\
&+ K E\int_{0}^{t} \int_{Z}  \Big|\int_{\kappa(n,s)}^s\int_Z\Gamma_3(x_{\kappa(n,r)}^n,z,z_1) \nu(dz_1) dr\Big|^p\nu(dz)ds \notag
\\
& + K n^{-\frac{p}{2}+1}E\int_{0}^{t} \int_{Z}  \int_{\kappa(n,s)}^s\int_Z|\Gamma_3(x_{\kappa(n,r)}^n,z,z_1)|^p \nu(dz_1)dr \nu(dz)ds \notag
\\
& \leq K E\int_{0}^{t} \int_{Z}  \int_{\kappa(n,s)}^s\int_Z|\Gamma_3(x_{\kappa(n,r)}^n,z,z_1)|^p \nu(dz_1)dr \nu(dz)ds \notag
\\
&+ K n^{-p+1} E\int_{0}^{t} \int_{Z}  \int_{\kappa(n,s)}^s\int_Z|\Gamma_3(x_{\kappa(n,r)}^n,z,z_1)|^p \nu(dz_1) dr\nu(dz)ds \notag
\end{align}
and then the application Remark \ref{rem:growth(n)} gives,
\begin{align}
F_{11} \leq K+K\int_0^t \sup_{0 \leq r \leq s} E|x_r^n|^{p_0} ds \label{eq:F11}
\end{align} 
for any $t \in [0,T]$ and $n\in\mathbb{N}$.
Therefore, on substituting all the estimates from equations \eqref{eq:F1} to \eqref{eq:F11} in equation \eqref{eq:F1+F11} yields,
\begin{align}
\sup_{0 \leq r \leq t} E|x_r^n|^{p_0} \leq K+K\int_0^t \sup_{0 \leq r \leq s} E|x_r^n|^{p_0} ds 
\end{align} 
for any $t \in [0,T]$ and $n\in\mathbb{N}$. The Gronwall's inequality completes the proof. 
\end{proof}
\section{Rate of Convergence} \label{sec:rate}
Before showing the main result of this article i.e. Theorem \ref{thm:main:thm}, one requires to prove the following lemmas.  
\begin{lem} \label{lem:one-step:rate:c}
Let Assumptions A-\ref{as:sde:ini} to A-\ref{as:sde:lip:der} be satisfied. Then, the following holds, 
$$
E\big(|x_t^n-x_{\kappa(n,t)}^n|^p|\mathscr{F}_{\kappa(n,t)}) \leq K n^{-1} (1+|x_{\kappa(n,t)}^n|)^p
$$
almost surely for any $2 \leq p \leq p_0/(2\rho+2)$ and $n\in\mathbb{N}$ where the positive constant $K$ does not depend on $n$.
\end{lem}
\begin{proof}
The proof follows by using the Milstein-type scheme defined in equation \eqref{eq:scheme}. 
\end{proof}
\begin{cor} \label{cor:one-step:rate:uc}
Let Assumptions A-\ref{as:sde:ini} to A-\ref{as:sde:lip:der} be satisfied. Then, the following holds, 
$$
\sup_{0 \leq t \leq T}E|x_t^n-x_{\kappa(n,t)}^n|^p \leq K n^{-1}
$$
for any $2 \leq p \leq p_0/(2\rho+2)$ and $n\in\mathbb{N}$ where the positive constant $K$ does not depend on $n$.
\end{cor}
\begin{proof}
The proof follows due to Lemmas [\ref{lem:one-step:rate:c}, \ref{lem:sde:moment:bound}].
\end{proof}
\begin{lem} \label{lem:gamma-gamma(n):rate}
Let Assumptions A-\ref{as:sde:ini} to A-\ref{as:sde:lip:der} be satisfied. Then, the following holds, 
\begin{align*}
\sup_{0 \leq t \leq T}E\int_Z|\gamma(x_t^n, z) -  \gamma^{(n)}(t,x_{\kappa(n,t)}^n, z)|^2 \nu(dz) \leq K n^{-2}
\end{align*}
for any $n\in\mathbb{N}$ where constant $K>0$ does not depend on $n$. 
\end{lem}
\begin{proof}
By It\^{o}'s formula, one obtains
\begin{align}
\gamma(x_t^n,z)&=\gamma(x_{\kappa(n,t)}^n,z)+\sum_{u=1}^{d}  \int_{\kappa(n,t)}^{t}\frac{\partial \gamma(x_s^n,z)}{\partial x^{u}}b^{(n),u}(x_{\kappa(n,s)}^n)ds \notag
\\
&+\frac{1}{2}\sum_{u}^{d} \sum_{v=1}^{d} \int_{\kappa(n,t)}^{t}\frac{\partial^2 \gamma(x_s^n,z)}{\partial x^{u} \partial x^{v}} \sum_{j=1}^{m}\sigma^{(n),uj}(s,x_{\kappa(n,s)}^n) \sigma^{(n),vj}(s,x_{\kappa(n,s)}^n)ds \notag
\\
&+ \sum_{u=1}^{d} \int_{\kappa(n,t)}^t \frac{\partial \gamma(x_s^n,z)}{\partial x^{u}} \sum_{j=1}^m\frac{\sigma^{uj}(x_{\kappa(n,s)}^n)}{1+n^{-1}|x_{\kappa(n,s)}^n|^{4 \rho}} dw_s^j \notag
\\
&+ \sum_{u=1}^{d} \int_{\kappa(n,t)}^t \frac{\partial \gamma(x_s^n,z)}{\partial x^{u}} \sum_{j=1}^m\int_{\kappa(n,s)}^s \sum_{k=1}^m \Lambda_1^{(n),k,uj}(x_{\kappa(n,r)}^n)dw_r^kdw_s^j\notag
\\
&+ \sum_{u=1}^{d} \int_{\kappa(n,t)}^t \frac{\partial \gamma(x_s^n,z)}{\partial x^{u}} \sum_{j=1}^m\int_{\kappa(n,s)}^s \int_Z \Lambda_2^{(n),uj}(x_{\kappa(n,r)}^n, z_1)\tilde{N}(dr,dz_1)dw_s^j\notag
\\
&+ \sum_{u=1}^{d} \int_{\kappa(n,t)}^t \frac{\partial \gamma(x_s^n,z)}{\partial x^{u}}\sum_{j=1}^m \int_{\kappa(n,s)}^s \int_Z \Lambda_3^{(n),uj}(x_{\kappa(n,r)}^n, z_1)N(dr,dz_1)dw_s^j \notag
\\
&+ \sum_{u=1}^{d} \int_{\kappa(n,t)}^t  \frac{\partial \gamma(x_s^n,z)}{\partial x^{u}} \int_Z \gamma^{u}(x_{\kappa(n,s)}^n,z_1) \tilde{N}(ds,dz_1) \notag
\\
&+ \sum_{u=1}^{d} \int_{\kappa(n,t)}^t \frac{\partial \gamma(x_s^n,z)}{\partial x^{u}} \int_{\kappa(n,s)}^s\int_Z  \sum_{k=1}^m \Gamma_1^{(n),k,u}(x_{\kappa(n,r)}^n,z_1)dw_r^k \tilde{N}(ds,dz_1) \notag
\\
&+ \sum_{u=1}^{d} \int_{\kappa(n,t)}^t  \frac{\partial \gamma(x_s^n,z)}{\partial x^{u}} \int_{\kappa(n,s)}^s \int_Z \int_Z  \Gamma_2^{(n),u}(x_{\kappa(n,r)}^n,z_1,z_2)\tilde{N}(dr,dz_2)\tilde{N}(ds,dz_1) \notag
\\
&+ \sum_{u=1}^{d} \int_{\kappa(n,t)}^t  \frac{\partial \gamma(x_s^n,z)}{\partial x^{u}} \int_{\kappa(n,s)}^s\int_Z\int_Z  \Gamma_3^{(n),u}(x_{\kappa(n,r)}^n,z_1,z_2)N(dr,dz_2)\tilde{N}(ds,dz_1) \notag
\\
 +\int_{\kappa(n,t)}^t &\int_Z \Big\{\gamma(x_s^n+\gamma^{(n)}(s,x_{\kappa(n,s)}^n,z_1),z)-\gamma(x_s^{n},z)-\sum_{u=1}^d \frac{\partial \gamma(x_s^n,z)}{\partial x^{u}} \gamma^{(n),u}(s,x_{\kappa(n,s)}^n,z_1) \Big\} N(ds,dz_1) \notag
\end{align}
which can also be written as below, 
\begin{align}
\gamma(x_t^n,z)&=\gamma(x_{\kappa(n,t)}^n,z)+\sum_{u=1}^{d}  \int_{\kappa(n,t)}^{t}\frac{\partial \gamma(x_s^n,z)}{\partial x^{u}}b^{(n),u}(x_{\kappa(n,s)}^n)ds \notag
\\
&+\frac{1}{2}\sum_{u}^{d} \sum_{v=1}^{d} \sum_{j=1}^{m}\int_{\kappa(n,t)}^{t}\frac{\partial^2 \gamma(x_s^n,z)}{\partial x^{u} \partial x^{v}} \sigma^{(n),uj}(s,x_{\kappa(n,s)}^n) \sigma^{(n),vj}(s,x_{\kappa(n,s)}^n)ds \notag
\\
&+ \sum_{u=1}^{d}\sum_{j=1}^m \int_{\kappa(n,t)}^t \Big\{\frac{\partial \gamma(x_s^n,z)}{\partial x^{u}}-\frac{\partial \gamma(x_{\kappa(n,s)}^n,z)}{\partial x^{u}}\Big\} \frac{\sigma^{uj}(x_{\kappa(n,s)}^n)}{1+n^{-1}|x_{\kappa(n,s)}^n|^{4 \rho}} dw_s^j \notag
\\
&+  \int_{\kappa(n,t)}^t \sum_{j=1}^m \Gamma_1^{(n),j}(x_{\kappa(n,s)}^n,z) dw_s^j \notag
\\
&+ \sum_{u=1}^{d}\sum_{j=1}^m \int_{\kappa(n,t)}^t \frac{\partial \gamma(x_s^n,z)}{\partial x^{u}}\int_{\kappa(n,s)}^s \sum_{k=1}^m \Lambda_1^{(n),k,uj}(x_{\kappa(n,r)}^n)dw_r^kdw_s^j\notag
\\
&+ \sum_{u=1}^{d}\sum_{j=1}^m \int_{\kappa(n,t)}^t \frac{\partial \gamma(x_s^n,z)}{\partial x^{u}}\int_{\kappa(n,s)}^s \int_Z \Lambda_2^{(n),uj}(x_{\kappa(n,r)}^n, z_1)\tilde{N}(dr,dz_1)dw_s^j\notag
\\
&+ \sum_{u=1}^{d}\sum_{j=1}^m \int_{\kappa(n,t)}^t \frac{\partial \gamma(x_s^n,z)}{\partial x^{u}}\int_{\kappa(n,s)}^s \int_Z \Lambda_3^{(n),uj}(x_{\kappa(n,r)}^n, z_1)N(dr,dz_1)dw_s^j \notag
\\
&+ \sum_{u=1}^{d} \int_{\kappa(n,t)}^t  \Big\{\frac{\partial \gamma(x_s^n,z)}{\partial x^{u}}-\frac{\partial \gamma(x_{\kappa(n,s)}^n,z)}{\partial x^{u}}\Big\} \int_Z \gamma^{u}(x_{\kappa(n,s)}^n,z_1) \tilde{N}(ds,dz_1) \notag
\\
&+  \int_{\kappa(n,t)}^t \int_Z \Gamma_2(x_{\kappa(n,s)}^n,z,z_1) \tilde{N}(ds,dz_1) \notag
\\
&+ \sum_{u=1}^{d} \int_{\kappa(n,t)}^t  \frac{\partial \gamma(x_s^n,z)}{\partial x^{u}} \int_{\kappa(n,s)}^s\int_Z \sum_{k=1}^m \Gamma_1^{(n),k,u}(x_{\kappa(n,r)}^n,z_1)dw_r^k \tilde{N}(ds,dz_1) \notag
\\
&+ \sum_{u=1}^{d} \int_{\kappa(n,t)}^t \frac{\partial \gamma(x_s^n,z)}{\partial x^{u}} \int_{\kappa(n,s)}^s \int_Z\int_Z  \Gamma_2^{(n),u}(x_{\kappa(n,r)}^n,z_1,z_2)\tilde{N}(dr,dz_2)\tilde{N}(ds,dz_1) \notag
\\
&+ \sum_{u=1}^{d} \int_{\kappa(n,t)}^t \frac{\partial \gamma(x_s^n,z)}{\partial x^{u}} \int_{\kappa(n,s)}^s  \int_Z \int_Z  \Gamma_3^{(n),u}(x_{\kappa(n,r)}^n,z_1,z_2)N(dr,dz_2)\tilde{N}(ds,dz_1) \notag
\\
&+\int_{\kappa(n,t)}^t \int_Z  \Big\{\gamma(x_s^n+\gamma^{(n)}(s,x_{\kappa(n,s)}^n,z_1),z)-\gamma(x_s^{n},z)\notag
\\
&\qquad \qquad-\sum_{u=1}^d \frac{\partial \gamma(x_s^n,z)}{\partial x^{u}} \gamma^{(n),u}(s,x_{\kappa(n,s)}^n,z_1) \Big\} N(ds,dz_1) \notag
\\
&-\int_{\kappa(n,t)}^t \int_Z \Big\{\gamma(x_{\kappa(n,s)}^n+\gamma(x_{\kappa(n,s)}^n,z_1),z)-\gamma(x_{\kappa(n,s)}^n,z)\notag
\\
&\qquad \qquad-\sum_{u=1}^d\frac{\partial \gamma(x_{\kappa(n,s)}^n,z)}{\partial x^{u}}\gamma^{u}(x_{\kappa(n,s)}^n,z_1)\Big\} N(ds,dz_1) \notag
\\
&+\int_{\kappa(n,t)}^t \int_Z\Gamma_3(x_{\kappa(n,s)}^n,z,z_1) N(ds,dz_1) \notag
\end{align}
and hence one obtains the following estimates,  
\begin{align}
E&\int_Z|\gamma(x_t^n, z)  -  \gamma^{(n)}(t,x_{\kappa(n,t)}^n, z)|^2\nu(dz)  \leq K E\int_Z\Big|\sum_{u=1}^{d}  \int_{\kappa(n,t)}^{t}\frac{\partial \gamma(x_s^n,z)}{\partial x^{u}}b^{(n),u}(x_{\kappa(n,s)}^n)ds\Big|^2 \nu(dz) \notag 
\\
&+K E\int_Z\Big|\sum_{u}^{d} \sum_{v=1}^{d} \sum_{j=1}^{m}\int_{\kappa(n,t)}^{t} \frac{\partial^2 \gamma(x_s^n,z)}{\partial x^{u} \partial x^{v}} \sigma^{(n),uj}(s,x_{\kappa(n,s)}^n) \sigma^{(n),vj}(s,x_{\kappa(n,s)}^n)ds\Big|^2\nu(dz) \notag
\\
&+ K E \int_Z\Big|\sum_{u=1}^{d}\sum_{j=1}^m \int_{\kappa(n,t)}^t \Big\{\frac{\partial \gamma(x_s^n,z)}{\partial x^{u}}-\frac{\partial \gamma(x_{\kappa(n,s)}^n,z)}{\partial x^{u}}\Big\} \frac{\sigma^{uj}(x_{\kappa(n,s)}^n)}{1+n^{-1}|x_{\kappa(n,s)}^n|^{4 \rho}} dw_s^j\Big|^2 \nu(dz) \notag
\\
&+ K E \int_Z\Big|\sum_{u=1}^{d}\sum_{j=1}^m \int_{\kappa(n,t)}^t \frac{\partial \gamma(x_s^n,z)}{\partial x^{u}}\int_{\kappa(n,s)}^s \sum_{k=1}^m \Lambda_1^{(n),k,uj}(x_{\kappa(n,r)}^n)dw_r^kdw_s^j\Big|^2 \nu(dz) \notag
\\
&+ K E \int_Z \Big|\sum_{u=1}^{d}\sum_{j=1}^m \int_{\kappa(n,t)}^t \frac{\partial \gamma(x_s^n,z)}{\partial x^{u}}\int_{\kappa(n,s)}^s \int_Z \Lambda_2^{(n),uj}(x_{\kappa(n,r)}^n, z_1)\tilde{N}(dr,dz_1)dw_s^j\Big|^2 \nu(dz)\notag
\\
&+ K E \int_Z \Big|\sum_{u=1}^{d}\sum_{j=1}^m \int_{\kappa(n,t)}^t \frac{\partial \gamma(x_s^n,z)}{\partial x^{u}}\int_{\kappa(n,s)}^s \int_Z \Lambda_3^{(n),uj}(x_{\kappa(n,r)}^n, z_1)N(dr,dz_1)dw_s^j\Big|^2 \nu(dz) \notag
\\
&+ K E \int_Z \Big|\sum_{u=1}^{d} \int_{\kappa(n,t)}^t \Big\{\frac{\partial \gamma(x_s^n,z)}{\partial x^{u}}-\frac{\partial \gamma(x_{\kappa(n,s)}^n,z)}{\partial x^{u}}\Big\}  \int_Z\gamma^{u}(x_{\kappa(n,s)}^n,z_1) \tilde{N}(ds,dz_1)\Big|^2 \nu(dz) \notag
\\
&+ K E \int_Z \Big| \sum_{u=1}^{d} \int_{\kappa(n,t)}^t  \frac{\partial \gamma(x_s^n,z)}{\partial x^{u}} \int_Z \int_{\kappa(n,s)}^s\sum_{k=1}^m \Gamma_1^{(n),k,u}(x_{\kappa(n,r)}^n,z_1)dw_r^k \tilde{N}(ds,dz_1)\Big|^2 \nu(dz) \notag
\\
&+ K E \int_Z \Big| \sum_{u=1}^{d} \int_{\kappa(n,t)}^t  \frac{\partial \gamma(x_s^n,z)}{\partial x^{u}} \int_Z \int_{\kappa(n,s)}^s \int_Z \Gamma_2^{(n),u}(x_{\kappa(n,r)}^n,z_1,z_2)\tilde{N}(dr,dz_2)\tilde{N}(ds,dz_1)\Big|^2 \nu(dz) \notag
\\
&+ K E \int_Z \Big|\sum_{u=1}^{d} \int_{\kappa(n,t)}^t  \frac{\partial \gamma(x_s^n,z)}{\partial x^{u}} \int_Z \int_{\kappa(n,s)}^s\int_Z  \Gamma_3^{(n),u}(x_{\kappa(n,r)}^n,z_1,z_2)N(dr,dz_2)\tilde{N}(ds,dz_1)\Big|^2 \nu(dz) \notag
\\
&+K E \int_Z \Big|\int_{\kappa(n,t)}^t \int_Z  \Big\{\gamma(x_s^n+\gamma^{(n)}(s,x_{\kappa(n,s)}^n,z_1),z)-\gamma(x_{\kappa(n,s)}^n+\gamma(x_{\kappa(n,s)}^n,z_1),z)+\gamma(x_{\kappa(n,s)}^n,z)  \notag
\\
&-\gamma(x_s^{n},z)+\sum_{u=1}^d\frac{\partial \gamma(x_{\kappa(n,s)}^n,z)}{\partial x^{u}}\gamma^{u}(x_{\kappa(n,s)}^n,z_1)-\sum_{u=1}^d \frac{\partial \gamma(x_s^n,z)}{\partial x^{u}} \gamma^{(n),u}(s,x_{\kappa(n,s)}^n,z_1) \Big\} N(ds,dz_1) \Big|^2 \nu(dz)\notag
\end{align}
for any $t\in[0,T]$, $z\in Z$ and  $n\in\mathbb{N}$. Due to H\"{o}lder's inequality and an elementary inequality of stochastic integral, one obtains
\begin{align}
E&\int_Z|\gamma(x_t^n, z)  -  \gamma^{(n)}(t,x_{\kappa(n,t)}^n, z)|^2\nu(dz) \notag
\\
& \leq K n^{-1} E \int_{\kappa(n,t)}^{t} \sum_{u=1}^d |b^{(n),u}(x_{\kappa(n,s)}^n)|^2 \int_Z\Big|\frac{\partial \gamma(x_s^n,z)}{\partial x^{u}}\Big|^2 \nu(dz) ds  \notag 
\\
&+K n^{-1}E\sum_{u}^{d} \sum_{v=1}^{d} \sum_{j=1}^{m}\int_{\kappa(n,t)}^{t} | \sigma^{(n),uj}(s,x_{\kappa(n,s)}^n) \sigma^{(n),vj}(s,x_{\kappa(n,s)}^n)|^2 \int_Z \Big|\frac{\partial^2 \gamma(x_s^n,z)}{\partial x^{u} \partial x^{v}}\Big|^2\nu(dz)ds \notag
\\
&+ K E \sum_{u=1}^{d}\sum_{j=1}^m \int_{\kappa(n,t)}^t |\sigma^{uj}(x_{\kappa(n,s)}^n)|^2 \int_Z \Big|\frac{\partial \gamma(x_s^n,z)}{\partial x^{u}}-\frac{\partial \gamma(x_{\kappa(n,s)}^n,z)}{\partial x^{u}}\Big|^2 \nu(dz)  ds  \notag
\\
&+ K E\sum_{u=1}^{d}\sum_{j=1}^m \sum_{k=1}^m \int_{\kappa(n,t)}^t  \Big|\int_{\kappa(n,s)}^s  \Lambda_1^{(n),k,uj}(x_{\kappa(n,r)}^n)dw_r^k\Big|^2 \int_Z\Big|\frac{\partial \gamma(x_s^n,z)}{\partial x^{u}}\Big|^2   \nu(dz)ds \notag
\\
&+ K E  \sum_{u=1}^{d}\sum_{j=1}^m \int_{\kappa(n,t)}^t \Big| \int_{\kappa(n,s)}^s \int_Z \Lambda_2^{(n),uj}(x_{\kappa(n,r)}^n, z_1)\tilde{N}(dr,dz_1)\Big|^2 \int_Z\Big|\frac{\partial \gamma(x_s^n,z)}{\partial x^{u}}\Big|^2 \nu(dz)ds \notag
\\
&+ K E  \sum_{u=1}^{d}\sum_{j=1}^m \int_{\kappa(n,t)}^t \Big| \int_{\kappa(n,s)}^s \int_Z \Lambda_3^{(n),uj}(x_{\kappa(n,r)}^n, z_1)N(dr,dz_1)\Big|^2 \int_Z\Big|\frac{\partial \gamma(x_s^n,z)}{\partial x^{u}}\Big|^2 \nu(dz)ds \notag
\\
&+ K E  \sum_{u=1}^{d} \int_{\kappa(n,t)}^t \int_Z |\gamma^{u}(x_{\kappa(n,s)}^n,z_1)|^2 \nu(dz_1) \int_Z\Big|\frac{\partial \gamma(x_s^n,z)}{\partial x^{u}}-\frac{\partial \gamma(x_{\kappa(n,s)}^n,z)}{\partial x^{u}}\Big|^2 \nu(dz) ds \notag
\\
&+ K E   \sum_{u=1}^{d} \int_{\kappa(n,t)}^t   \int_Z \Big|\int_{\kappa(n,s)}^s\sum_{k=1}^m \Gamma_1^{(n),k,u}(x_{\kappa(n,r)}^n,z_1)dw_r^k\Big|^2 \nu(dz_1) \int_Z \Big|\frac{\partial \gamma(x_s^n,z)}{\partial x^{u}}\Big|^2 \nu(dz) ds \notag
\\
&+ K E   \sum_{u=1}^{d} \int_{\kappa(n,t)}^t   \int_Z \Big| \int_{\kappa(n,s)}^s \int_Z \Gamma_2^{(n),u}(x_{\kappa(n,r)}^n,z_1,z_2) \tilde{N}(dr,dz_2)\Big|^2 \nu(dz_1) \int_Z\Big|\frac{\partial \gamma(x_s^n,z)}{\partial x^{u}} \Big|^2 \nu(dz)  ds \notag
\\
&+ K E \sum_{u=1}^{d} \int_{\kappa(n,t)}^t  \int_Z\Big| \int_{\kappa(n,s)}^s\int_Z  \Gamma_3^{(n),u}(x_{\kappa(n,r)}^n,z_1,z_2)N(dr,dz_2)\Big|^2 \nu(dz_1)  \int_Z  \Big|  \frac{\partial \gamma(x_s^n,z)}{\partial x^{u}}\Big|^2 \nu(dz) ds \notag
\\
&+K E \int_Z \int_{\kappa(n,t)}^t \int_Z  \Big|\gamma(x_s^n+\gamma^{(n)}(s,x_{\kappa(n,s)}^n,z_1),z)-\gamma(x_{\kappa(n,s)}^n+\gamma(x_{\kappa(n,s)}^n,z_1),z)\notag
\\
&\qquad+\gamma(x_{\kappa(n,s)}^n,z)-\gamma(x_s^{n},z) \notag
\\
&+\sum_{u=1}^d\frac{\partial \gamma(x_{\kappa(n,s)}^n,z)}{\partial x^{u}}\gamma^{u}(x_{\kappa(n,s)}^n,z_1)-\sum_{u=1}^d \frac{\partial \gamma(x_s^n,z)}{\partial x^{u}} \gamma^{(n),u}(s,x_{\kappa(n,s)}^n,z_1) \Big|^2 \nu(dz_1) ds  \nu(dz)\notag
\\
=:& \,G_1+G_2+G_3+G_4+G_5+G_6+G_7+G_8+G_9+G_{10}+G_{11} \label{G1+G11}
\end{align}
for any $t \in [0,T]$ and $n\in\mathbb{N}$. For estimating $G_1$, one uses Remarks [\ref{rem:growth}, \ref{rem:growth(n)}] and Lemma \ref{lem:mb:scheme} to obtain the following estimates, 
\begin{align}
G_1&:=n^{-1} E \int_{\kappa(n,t)}^{t} \sum_{u=1}^d |b^{(n),u}(x_{\kappa(n,s)}^n)|^2 \int_Z\Big|\frac{\partial \gamma(x_s^n,z)}{\partial x^{u}}\Big|^2 \nu(dz) ds  \notag 
\\
&\leq K n^{-2} E (1+|x_{\kappa(n,t)}^n|)^{2\rho+2} \leq K n^{-2}  \label{eq:G1} 
\end{align} 
for any $t\in[0,T]$ and $n\in\mathbb{N}$. Similarly, $G_2$ can be estimated by using Remark [\ref{rem:growth}, \ref{rem:growth(n)}] and Lemma \ref{lem:mb:scheme}, 
\begin{align}
G_2&:=K n^{-1}E\sum_{u}^{d} \sum_{v=1}^{d} \sum_{j=1}^{m}\int_{\kappa(n,t)}^{t} | \sigma^{(n),uj}(s,x_{\kappa(n,s)}^n) \sigma^{(n),vj}(s,x_{\kappa(n,s)}^n)|^2 \int_Z \Big|\frac{\partial^2 \gamma(x_s^n,z)}{\partial x^{u} \partial x^{v}}\Big|^2\nu(dz)ds \notag
\\
&\leq K n^{-1}E\int_{\kappa(n,t)}^{t} |\sigma^{(n)}(s,x_{\kappa(n,s)}^n)|^4 ds\leq K n^{-2} \label{eq:G2}
\end{align}
for any $t\in [0,T]$ and $n\in\mathbb{N}$. Also, for estimating $G_3$, one uses Assumption A-\ref{as:sde:lip:der}, Remark \ref{rem:growth}  and Lemma \ref{lem:one-step:rate:c} to write, 
\begin{align}
G_3&:=K E \sum_{u=1}^{d}\sum_{j=1}^m \int_{\kappa(n,t)}^t |\sigma^{uj}(x_{\kappa(n,s)}^n)|^2 \int_Z \Big|\frac{\partial \gamma(x_s^n,z)}{\partial x^{u}}-\frac{\partial \gamma(x_{\kappa(n,s)}^n,z)}{\partial x^{u}}\Big|^2 \nu(dz)  ds  \notag
\\
&\leq K E \int_{\kappa(n,t)}^t |x_s^n-x_{\kappa(n,s)}^n|^2 |\sigma(x_{\kappa(n,s)}^n)|^2 ds \leq Kn^{-2}  \label{eq:G3}
\end{align}
for any $t\in[0,T]$ and $n\in\mathbb{N}$. Due to an elementary inequality of stochastic integral, Remarks [\ref{rem:growth}, \ref{rem:growth(n)}] and Lemma \ref{lem:mb:scheme}, $G_4$ is estimated by, 
\begin{align}
G_4&:=K E\sum_{u=1}^{d}\sum_{j=1}^m \sum_{k=1}^m \int_{\kappa(n,t)}^t  \Big|\int_{\kappa(n,s)}^s  \Lambda_1^{(n),k,uj}(x_{\kappa(n,r)}^n)dw_r^k\Big|^2 \int_Z\Big|\frac{\partial \gamma(x_s^n,z)}{\partial x^{u}}\Big|^2   \nu(dz)ds \notag
\\
&\leq K E\sum_{u=1}^{d}\sum_{j=1}^m \sum_{k=1}^m \int_{\kappa(n,t)}^t  \Big|\int_{\kappa(n,s)}^s  \Lambda_1^{(n),k,uj}(x_{\kappa(n,r)}^n)dw_r^k\Big|^2 ds \notag
\\
&\leq K E \sum_{k=1}^m \int_{\kappa(n,t)}^t \int_{\kappa(n,s)}^s  |\Lambda_1^{(n),k}(x_{\kappa(n,r)}^n)|^2drds \leq Kn^{-2} \label{eq:G4}
\end{align}
for any $t\in[0,T]$ and $n\in\mathbb{N}$. Furthermore, one uses Remarks [\ref{rem:growth}, \ref{rem:growth(n)}] and Lemma \ref{lem:mb:scheme} to estimate $G_5$ as,  
\begin{align}
G_5&:=K E  \sum_{u=1}^{d}\sum_{j=1}^m \int_{\kappa(n,t)}^t \Big| \int_{\kappa(n,s)}^s \int_Z \Lambda_2^{(n),uj}(x_{\kappa(n,r)}^n, z_1)\tilde{N}(dr,dz_1)\Big|^2 \int_Z\Big|\frac{\partial \gamma(x_s^n,z)}{\partial x^{u}}\Big|^2 \nu(dz)ds \notag
\\
&\leq K E  \sum_{u=1}^{d}\sum_{j=1}^m \int_{\kappa(n,t)}^t \Big| \int_{\kappa(n,s)}^s \int_Z \Lambda_2^{(n),uj}(x_{\kappa(n,r)}^n, z_1)\tilde{N}(dr,dz_1)\Big|^2 ds \notag
\\
&\leq K E \int_{\kappa(n,t)}^t \int_{\kappa(n,s)}^s \int_Z |\Lambda_2^{(n)}(x_{\kappa(n,r)}^n, z_1)|^2 \nu(dz_1)drds \leq Kn^{-2} \label{eq:G5}
\end{align}
for any $t \in [0,T]$ and $n\in\mathbb{N}$. Similarly, one estimates $G_6$ as,  
\begin{align}
G_6&:=K E  \sum_{u=1}^{d}\sum_{j=1}^m \int_{\kappa(n,t)}^t \Big| \int_{\kappa(n,s)}^s \int_Z \Lambda_3^{(n),uj}(x_{\kappa(n,r)}^n, z_1)N(dr,dz_1)\Big|^2 \int_Z\Big|\frac{\partial \gamma(x_s^n,z)}{\partial x^{u}}\Big|^2 \nu(dz)ds \notag
\\
& \leq K E  \sum_{u=1}^{d}\sum_{j=1}^m \int_{\kappa(n,t)}^t \Big| \int_{\kappa(n,s)}^s \int_Z \Lambda_3^{(n),uj}(x_{\kappa(n,r)}^n, z_1)\tilde{N}(dr,dz_1)\Big|^2 ds \notag
\\
& + K E  \sum_{u=1}^{d}\sum_{j=1}^m \int_{\kappa(n,t)}^t \Big| \int_{\kappa(n,s)}^s \int_Z \Lambda_3^{(n),uj}(x_{\kappa(n,r)}^n, z_1)\nu(dz_1)dr\Big|^2 ds \notag
\\
&\leq K E  \int_{\kappa(n,t)}^t \int_{\kappa(n,s)}^s\int_Z |\Lambda_3^{(n)}(x_{\kappa(n,r)}^n, z_1)|^2\nu(z_1)drds\leq Kn^{-2}  \label{eq:G6}
\end{align}
for any $t \in [0,T]$ and $n\in\mathbb{N}$. Similarly, for estimating $G_7$, one uses Remark \ref{rem:growth}, Assumption \ref{as:sde:lip:der} and Lemma \ref{lem:one-step:rate:c} to obtain the following, 
\begin{align}
G_7&:=K E  \sum_{u=1}^{d} \int_{\kappa(n,t)}^t \int_Z |\gamma^{u}(x_{\kappa(n,s)}^n,z_1)|^2 \nu(dz_1) \int_Z\Big|\frac{\partial \gamma(x_s^n,z)}{\partial x^{u}}-\frac{\partial \gamma(x_{\kappa(n,s)}^n,z)}{\partial x^{u}}\Big|^2 \nu(dz) ds \notag
\\
&\leq K E  \int_{\kappa(n,t)}^t (1+|x_{\kappa(n,s)}^n|)^2 |x_s^n-x_{\kappa(n,s)}^n|^2  ds \leq K n^{-2} \label{eq:G7}
\end{align}
for any $t\in [0,T]$ and $n\in\mathbb{N}$. Also, ones the similar approach as before to estimate $G_8$ by,  
\begin{align}
G_8&:=K E   \sum_{u=1}^{d} \int_{\kappa(n,t)}^t   \int_Z \Big|\int_{\kappa(n,s)}^s\sum_{k=1}^m \Gamma_1^{(n),k,u}(x_{\kappa(n,r)}^n,z_1)dw_r^k\Big|^2 \nu(dz_1) \int_Z \Big|\frac{\partial \gamma(x_s^n,z)}{\partial x^{u}}\Big|^2 \nu(dz) ds \notag
\\
&\leq K E   \sum_{u=1}^{d} \int_{\kappa(n,t)}^t   \int_Z \Big|\int_{\kappa(n,s)}^s\sum_{k=1}^m \Gamma_1^{(n),k,u}(x_{\kappa(n,r)}^n,z_1)dw_r^k\Big|^2 \nu(dz_1)  ds \notag
\\
&\leq K E  \int_{\kappa(n,t)}^t \int_Z  \int_{\kappa(n,s)}^s  |\Gamma_1^{(n)}(x_{\kappa(n,r)}^n,z_1)|^2 dr \nu(dz_1) ds \leq Kn^{-2} \label{eq:G8}
\end{align}
for any $t \in [0,T]$ and $n\in\mathbb{N}$. Also, one estimates $G_9$ by, 
\begin{align}
G_9&:=K E   \sum_{u=1}^{d} \int_{\kappa(n,t)}^t   \int_Z \Big| \int_{\kappa(n,s)}^s \int_Z \Gamma_2^{(n),u}(x_{\kappa(n,r)}^n,z_1,z_2) \tilde{N}(dr,dz_2)\Big|^2 \nu(dz_1) \int_Z\Big|\frac{\partial \gamma(x_s^n,z)}{\partial x^{u}} \Big|^2 \nu(dz)  ds \notag
\\
&\leq K E   \sum_{u=1}^{d} \int_{\kappa(n,t)}^t   \int_Z \Big| \int_{\kappa(n,s)}^s \int_Z \Gamma_2^{(n),u}(x_{\kappa(n,r)}^n,z_1,z_2) \tilde{N}(dr,dz_2)\Big|^2 \nu(dz_1)   ds \notag
\\
&\leq K E  \int_{\kappa(n,t)}^t \int_Z \int_{\kappa(n,s)}^s\int_Z  |\Gamma_2(x_{\kappa(n,r)}^n,z_1,z_2)|^2 \nu(dz_2) dr \nu(dz_1)ds  \leq K n^{-2} \label{eq:G9}
\end{align}
for any $t\in [0,T]$ and $n\in\mathbb{N}$. Moreover, for estimating $G_{10}$, one writes,
\begin{align}
G_{10}&:=K E \sum_{u=1}^{d} \int_{\kappa(n,t)}^t  \int_Z\Big| \int_{\kappa(n,s)}^s\int_Z  \Gamma_3^{(n),u}(x_{\kappa(n,r)}^n,z_1,z_2)N(dr,dz_2)\Big|^2 \nu(dz_1)  \int_Z  \Big|  \frac{\partial \gamma(x_s^n,z)}{\partial x^{u}}\Big|^2 \nu(dz) ds \notag
\\
&\leq K E \sum_{u=1}^{d} \int_{\kappa(n,t)}^t  \int_Z\Big| \int_{\kappa(n,s)}^s\int_Z  \Gamma_3^{(n),u}(x_{\kappa(n,r)}^n,z_1,z_2)N(dr,dz_2)\Big|^2 \nu(dz_1)  ds \notag
\\
&\leq K E \int_{\kappa(n,t)}^t \int_Z  \int_{\kappa(n,s)}^s\int_Z  |\Gamma_3^{(n)}(x_{\kappa(n,r)}^n,z_1,z_2)|^2 \nu(dz_2) dr \nu(dz_1)ds  \leq K n^{-2} \label{eq:G10}
\end{align} 
for any $t\in [0,T]$. Finally, for estimating $G_{11}$, one writes, 
\begin{align}
G_{11}&:=K E \int_Z \int_{\kappa(n,t)}^t \int_Z  \Big|\gamma(x_s^n+\gamma^{(n)}(s,x_{\kappa(n,s)}^n,z_1),z)-\gamma(x_{\kappa(n,s)}^n+\gamma(x_{\kappa(n,s)}^n,z_1),z)\notag
\\
&\qquad+\gamma(x_{\kappa(n,s)}^n,z)-\gamma(x_s^{n},z) \notag
\\
&+\sum_{u=1}^d\frac{\partial \gamma(x_{\kappa(n,s)}^n,z)}{\partial x^{u}}\gamma^{u}(x_{\kappa(n,s)}^n,z_1)-\sum_{u=1}^d \frac{\partial \gamma(x_s^n,z)}{\partial x^{u}} \gamma^{(n),u}(s,x_{\kappa(n,s)}^n,z_1) \Big|^2 \nu(dz_1) ds  \nu(dz)\notag
\\
&\leq K E \int_Z \int_{\kappa(n,t)}^t\Big\{ \int_Z  |\gamma(x_s^n+\gamma^{(n)}(s,x_{\kappa(n,s)}^n,z_1),z)-\gamma(x_{\kappa(n,s)}^n+\gamma(x_{\kappa(n,s)}^n,z_1),z)|^2 \nu(dz) \notag
\\
&\qquad+ \int_Z |\gamma(x_{\kappa(n,s)}^n,z)-\gamma(x_s^{n},z)|^2\nu(dz) \notag
\\
&+\int_Z\sum_{u=1}^d\Big|\frac{\partial \gamma(x_{\kappa(n,s)}^n,z)}{\partial x^{u}}- \frac{\partial \gamma(x_s^n,z)}{\partial x^{u}}\Big|^2 \nu(dz)|\gamma^u(x_{\kappa(n,s)}^n,z_1)|^2 \Big\}\nu(dz_1) ds  \notag
\\
&+K E \int_Z \int_{\kappa(n,t)}^t \int_Z\sum_{u=1}^d\Big|\frac{\partial \gamma(x_s^n,z)}{\partial x^{u}}\Big|^2\Big|\int_{\kappa(n,s)}^{s} \sum_{k=1}^{m}\Gamma_1^{(n),j,u}(x_{\kappa(n,r)}^n,z_1)dw_r^k\Big|^2 \nu(dz_1)ds\nu(dz) \notag
\\
&+K E \int_Z \int_{\kappa(n,t)}^t \int_Z\sum_{u=1}^d\Big|\frac{\partial \gamma(x_s^n,z)}{\partial x^{u}}\Big|^2\Big|\int_{\kappa(n,s)}^{s} \int_Z \Gamma_2^u(x_{\kappa(n,r)}^n,z_1,z_2)\tilde{N}(dr,dz_2)\Big|^2 \nu(dz_1)ds\nu(dz) \notag
\\
&+K E \int_Z \int_{\kappa(n,t)}^t \int_Z\sum_{u=1}^d\Big|\frac{\partial \gamma(x_s^n,z)}{\partial x^{u}}\Big|^2\Big|\int_{\kappa(n,s)}^{s} \int_Z \Gamma_3^u(x_{\kappa(n,r)}^n,z_1,z_2)N(dr,dz_2)\Big|^2 \nu(dz_1)ds\nu(dz) \notag
\end{align}
which due to Remark [\ref{rem:growth}, \ref{rem:growth(n)}], Lemma \ref{lem:one-step:rate:c}, Corollary \ref{cor:one-step:rate:uc} and an elementary inequality of stochastic integrals gives,  
\begin{align}
G_{11} &\leq K E \int_{\kappa(n,t)}^t|x_s^n-x_{\kappa(n,s)}^n|^2 +K E \int_{\kappa(n,t)}^t|x_s^n-x_{\kappa(n,s)}^n|^2(1+|x_{\kappa(n,s)}^n|^2)ds  \notag
\\
&+K E \int_Z \int_{\kappa(n,t)}^t \int_{\kappa(n,s)}^{s} |\Gamma_1^{(n)}(x_{\kappa(n,r)}^n,z_1)|^2 dr \nu(dz_1)ds \notag
\\
&+K E  \int_{\kappa(n,t)}^t \int_Z \int_{\kappa(n,s)}^{s} \int_Z |\Gamma_2(x_{\kappa(n,r)}^n,z_1,z_2)|^2dr\nu(dz_2) \nu(dz_1)ds \notag
\\
&+K E \int_{\kappa(n,t)}^t \int_Z\int_{\kappa(n,s)}^{s} \int_Z |\Gamma_3(x_{\kappa(n,r)}^n,z_1,z_2)|^2 dr\nu(dz_2)\nu(dz_1)ds \leq Kn^{-2} \label{eq:G11}
\end{align}
for any $t\in[0,T]$ and $n\in\mathbb{N}$. The proof is completed by substituting estimates from \eqref{eq:G1} to \eqref{eq:G11} in equation \eqref{G1+G11}. 
\end{proof}
\begin{lem} \label{lem:sigma-sigma(n):rate}
Let Assumptions A-\ref{as:sde:ini} to A-\ref{as:sde:lip:der} be satisfied. Then, the following holds, 
\begin{align*}
\sup_{0 \leq t \leq T}E|\sigma(x_t^n) -  \sigma^{(n)}(t,x_{\kappa(n,t)}^n)|^2  \leq K n^{-1-\frac{2}{2+\delta}}.
\end{align*}
for any $n\in\mathbb{N}$ where the positive constant $K$ does not depend on $n$.  
\end{lem}
\begin{proof}
By the application of It\^{o}'s formula, one obtains
\begin{align}
\sigma(x_t^n)&=\sigma(x_{\kappa(n,t)}^n)+\sum_{u=1}^{d}  \int_{\kappa(n,t)}^{t}\frac{\partial \sigma(x_s^n)}{\partial x^{u}}b^{(n),u}(x_{\kappa(n,s)}^n)ds \notag
\\
&+\frac{1}{2}\sum_{u}^{d} \sum_{v=1}^{d} \int_{\kappa(n,t)}^{t}\frac{\partial^2 \sigma(x_s^n)}{\partial x^{u} \partial x^{v}} \sum_{j=1}^{m}\sigma^{(n),uj}(s,x_{\kappa(n,s)}^n) \sigma^{(n),vj}(s,x_{\kappa(n,s)}^n)ds \notag
\\
&+ \sum_{u=1}^{d} \int_{\kappa(n,t)}^t \frac{\partial \sigma(x_s^n)}{\partial x^{u}} \sum_{j=1}^m\frac{\sigma^{uj}(x_{\kappa(n,s)}^n)}{1+n^{-1}|x_{\kappa(n,s)}^n|^{4\rho}} dw_s^j \notag
\\
&+ \sum_{u=1}^{d} \int_{\kappa(n,t)}^t \frac{\partial \sigma(x_s^n)}{\partial x^{u}} \sum_{j=1}^m \int_{\kappa(n,s)}^s \sum_{k=1}^m \Lambda_1^{(n),k, uj}(x_{\kappa(n,r)}^n)dw_r^k dw_s^j \notag
\\
&+ \sum_{u=1}^{d} \int_{\kappa(n,t)}^t \frac{\partial \sigma(x_s^n)}{\partial x^{u}} \sum_{j=1}^m \int_{\kappa(n,s)}^s \int_Z \Lambda_2^{(n),uj}(x_{\kappa(n,r)}^n,z_1)\tilde{N}(dr,dz_1) dw_s^j \notag
\\
&+ \sum_{u=1}^{d} \int_{\kappa(n,t)}^t \frac{\partial \sigma(x_s^n)}{\partial x^{u}} \sum_{j=1}^m \int_{\kappa(n,s)}^s \int_Z \Lambda_3^{(n),uj}(x_{\kappa(n,r)}^n,z_1)N(dr,dz_1) dw_s^j \notag
\\
&+ \sum_{u=1}^{d} \int_{\kappa(n,t)}^t \int_Z \frac{\partial \sigma(x_s^n)}{\partial x^{u}} \gamma^{u}(x_{\kappa(n,s)}^n,z_1) \tilde{N}(ds,dz_1) \notag
\\
&+ \sum_{u=1}^{d} \int_{\kappa(n,t)}^t \int_Z \frac{\partial \sigma(x_s^n)}{\partial x^{u}} \int_{\kappa(n,s)}^s \sum_{k=1}^{m} \Gamma_1^{(n),k, u}(x_{\kappa(n,r)}^n, z_1)dw_r^k \tilde{N}(ds,dz_1) \notag
\\
&+ \sum_{u=1}^{d} \int_{\kappa(n,t)}^t \int_Z \frac{\partial \sigma(x_s^n)}{\partial x^{u}} \int_{\kappa(n,s)}^s \int_Z \Gamma_2^u(x_{\kappa(n,r)}^n,z_1,z_2) \tilde{N}(dr,dz_2) \tilde{N}(ds,dz_1) \notag
\\
&+ \sum_{u=1}^{d} \int_{\kappa(n,t)}^t \int_Z \frac{\partial \sigma(x_s^n)}{\partial x^{u}} \int_{\kappa(n,s)}^s \int_Z \Gamma_3^u(x_{\kappa(n,r)}^n,z_1,z_2)N(dr,dz_2) \tilde{N}(ds,dz_1) \notag
\\
+\int_{\kappa(n,t)}^t \int_Z & \big\{\sigma(x_s^n+\gamma^{(n)}(s,x_{\kappa(n,s)}^n,z_1))-\sigma(x_s^{n})-\sum_{u=1}^d \frac{\partial \sigma(x_s^n)}{\partial x^{u}} \gamma^{(n),u}(s,x_{\kappa(n,s)}^n,z_1) \big\} N(ds,dz_1) \notag
\end{align}
which on using the definition \eqref{eq:sigma(n)} gives, 
\begin{align}
E|&\sigma(x_t^n)-\sigma^{(n)}(t,x_{\kappa(n,t)}^n)|^2\leq KE|\sigma(x_{\kappa(n,t)}^n)-\Lambda_0^{(n)}(x_{\kappa(n,t)}^n)|^2 \notag
\\
& + K E\Big|\sum_{u=1}^{d}  \int_{\kappa(n,t)}^{t}\frac{\partial \sigma(x_s^n)}{\partial x^{u}}b^{(n),u}(x_{\kappa(n,s)}^n)ds\Big|^2 \notag
\\
&+KE\Big|\sum_{u}^{d} \sum_{v=1}^{d} \sum_{j=1}^{m}\int_{\kappa(n,t)}^{t}\frac{\partial^2 \sigma(x_s^n)}{\partial x^{u} \partial x^{v}} \sigma^{(n),uj}(s,x_{\kappa(n,s)}^n) \sigma^{(n),vj}(s,x_{\kappa(n,s)}^n)ds\Big|^2 \notag
\\
&+ KE\Big| \sum_{k=1}^m \int_{\kappa(n,t)}^t \sum_{u=1}^{d}\Big\{\frac{\partial \sigma(x_s^n)}{\partial x^{u}}-\frac{\partial \sigma(x_{\kappa(n,s)}^n)}{\partial x^{u}}\Big\} \frac{\sigma^{uk}(x_{\kappa(n,s)}^n)}{1+n^{-1}|x_{\kappa(n,s)}^n|^{4\rho}} dw_s^k\Big|^2 \notag
\\
&+ K E\Big|\sum_{u=1}^{d}\sum_{j=1}^m \int_{\kappa(n,t)}^t \frac{\partial \sigma(x_s^n)}{\partial x^{u}} \int_{\kappa(n,s)}^s \sum_{k=1}^m \Lambda_1^{(n),k, uj}(x_{\kappa(n,r)}^n)dw_r^k dw_s^j\Big|^2 \notag
\\
&+ K E\Big|\sum_{u=1}^{d}\sum_{j=1}^m \int_{\kappa(n,t)}^t \frac{\partial \sigma(x_s^n)}{\partial x^{u}} \int_{\kappa(n,s)}^s \int_Z \Lambda_2^{(n),uj}(x_{\kappa(n,r)}^n,z_1)\tilde{N}(dr,dz_1) dw_s^j\Big|^2 \notag
\\
&+ K E\Big|\sum_{u=1}^{d}\sum_{j=1}^m \int_{\kappa(n,t)}^t \frac{\partial \sigma(x_s^n)}{\partial x^{u}}  \int_{\kappa(n,s)}^s \int_Z \Lambda_3^{(n),uj}(x_{\kappa(n,r)}^n,z_1)N(dr,dz_1) dw_s^j\Big|^2 \notag
\\
&+ K E\Big|\sum_{u=1}^{d} \int_{\kappa(n,t)}^t \int_Z \Big\{\frac{\partial \sigma(x_s^n)}{\partial x^{u}}-\frac{\partial \sigma(x_{\kappa(n,s)}^n)}{\partial x^{u}}\Big\} \gamma^{u}(x_{\kappa(n,s)}^n,z_1) \tilde{N}(ds,dz_1)\Big|^2 \notag
\\
&+ K E\Big| \sum_{u=1}^{d} \int_{\kappa(n,t)}^t \int_Z \Big\{1 -\frac{1}{1+n^{-1}|x_{\kappa(n,s)}^n|^{4\rho}}\Big\} \frac{\partial \sigma(x_{\kappa(n,s)}^n)}{\partial x^{u}} \gamma^{u}(x_{\kappa(n,s)}^n,z_1) \tilde{N}(ds,dz_1)\Big|^2 \notag
\\
&+ K E\Big| \sum_{u=1}^{d} \int_{\kappa(n,t)}^t \int_Z \frac{\partial \sigma(x_s^n)}{\partial x^{u}} \int_{\kappa(n,s)}^s \sum_{k=1}^{m} \Gamma_1^{(n),k, u}(x_{\kappa(n,r)}^n, z_1)dw_r^k \tilde{N}(ds,dz_1)\Big|^2 \notag
\\
&+ K E\Big| \sum_{u=1}^{d} \int_{\kappa(n,t)}^t \int_Z \frac{\partial \sigma(x_s^n)}{\partial x^{u}} \int_{\kappa(n,s)}^s \int_Z \Gamma_2^u(x_{\kappa(n,r)}^n,z_1,z_2) \tilde{N}(dr,dz_2) \tilde{N}(ds,dz_1) \Big|^2 \notag
\\
&+  K E\Big| \sum_{u=1}^{d} \int_{\kappa(n,t)}^t \int_Z \frac{\partial \sigma(x_s^n)}{\partial x^{u}} \int_{\kappa(n,s)}^s \int_Z \Gamma_3^u(x_{\kappa(n,r)}^n,z_1,z_2)N(dr,dz_2) \tilde{N}(ds,dz_1) \Big|^2\notag
\\
&+K E\Big| \int_{\kappa(n,t)}^t \int_Z  \Big\{\sigma(x_s^n+\gamma^{(n)}(s,x_{\kappa(n,s)}^n,z_1))-\sigma(x_s^{n})-\sum_{u=1}^d \frac{\partial \sigma(x_s^n)}{\partial x^{u}} \gamma^{(n),u}(s,x_{\kappa(n,s)}^n,z_1)  \notag
\\
& - \frac{\sigma(x_s^n+\gamma(x_{\kappa(n,s)}^n,z_1))-\sigma(x_s^{n})-\sum_{u=1}^d \frac{\partial \sigma(x_s^n)}{\partial x^{u}} \gamma^{u}(x_{\kappa(n,s)}^n,z_1) }{1+n^{-1}|x_{\kappa(n,s)}^n|^{4\rho}}\Big\} N(ds,dz_1)\Big|^2 \notag
\\
:=& H_1+H_2+H_3+H_4+H_5+H_6+H_7+H_8+H_9+H_{10}+H_{11}+H_{12}+H_{13} \label{eq:H1+H13} 
\end{align}
for any $t\in [0,T]$ and $n\in\mathbb{N}$. One uses Remark \ref{rem:growth} and Lemma \ref{lem:mb:scheme} to obtain,  
\begin{align}
H_1&:=KE|\sigma(x_{\kappa(n,t)}^n)-\Lambda_0^{(n)}(x_{\kappa(n,t)}^n)|^2 \notag
\\
&\leq KE\Big|\sigma(x_{\kappa(n,t)}^n)-\frac{\sigma(x_{\kappa(n,t)}^n)}{1+n^{-1}|x_{\kappa(n,t)}^n|^{2\rho}}\Big|^2 \notag 
\\
&\leq K n^{-2}E|\sigma(x_{\kappa(n,t)}^n)|^2|x_{\kappa(n,t)}^n|^{4\rho} \leq Kn^{-2} \label{eq:H1}
\end{align}
for any $t\in[0,T]$ and $n\in\mathbb{N}$. For estimating $H_2$, one applies H\"{o}lder's inequality, Remarks [\ref{rem:growth}, \ref{rem:growth(n)}] and Lemma \ref{lem:mb:scheme} to get, 
\begin{align}
H_2&:=K  E  \Big|\sum_{u=1}^{d}\int_{\kappa(n,t)}^{t}\frac{\partial \sigma(x_s^n)}{\partial x^{u}}b^{(n),u}(x_{\kappa(n,s)}^n) ds\Big|^2 \notag
\\
&\leq K  n^{-1}E \int_{\kappa(n,t)}^{t} (1+|x_s^n|)^\rho (1+|x_{\kappa(n,s)}^n|)^{2\rho+2} ds \leq K n^{-2} \label{eq:H2}
\end{align}
for any $t\in [0,T]$ and $n\in\mathbb{N}$. Similarly, one estimates $H_3$ by,  
\begin{align}
H_3&:=KE\Big|\sum_{u}^{d} \sum_{v=1}^{d} \sum_{j=1}^{m}\int_{\kappa(n,t)}^{t}\frac{\partial^2 \sigma(x_s^n)}{\partial x^{u} \partial x^{v}} \sigma^{(n),uj}(s,x_{\kappa(n,s)}^n) \sigma^{(n),vj}(s,x_{\kappa(n,s)}^n)ds\Big|^2 \notag
\\
&\leq Kn^{-1}E\int_{\kappa(n,t)}^{t}(1+|x_s^n|)^{\rho-2} |\sigma^{(n)}(s,x_{\kappa(n,s)}^n)|^4ds  \leq K n^{-2} \label{eq:H3}
\end{align}
for any $t\in[0,T]$ and $n\in\mathbb{N}$. For estimating $H_4$, one uses  Assumption A-\ref{as:sde:lip:der}, Remark \ref{rem:growth} to obtain,   
\begin{align}
H_4&:=KE\Big| \sum_{k=1}^m \int_{\kappa(n,t)}^t \sum_{u=1}^{d}\Big\{\frac{\partial \sigma(x_s^n)}{\partial x^{u}}-\frac{\partial \sigma(x_{\kappa(n,s)}^n)}{\partial x^{u}}\Big\} \frac{\sigma^{uk}(x_{\kappa(n,s)}^n)}{1+n^{-1}|x_{\kappa(n,s)}^n|^{4\rho}} dw_s^k\Big|^2 \notag
\\
&\leq  KE\sum_{k=1}^m \int_{\kappa(n,t)}^t \sum_{u=1}^{d}\Big|\frac{\partial \sigma(x_s^n)}{\partial x^{u}}-\frac{\partial \sigma(x_{\kappa(n,s)}^n)}{\partial x^{u}}\Big|^2 |\sigma^{uk}(x_{\kappa(n,s)}^n)|^2 ds \notag
\\
&\leq KE \int_{\kappa(n,t)}^t (1+|x_s^n|+|x_{\kappa(n,s)}^n|)^{\rho-2}|x_s^n-x_{\kappa(n,s)}^n|^2 (1+|x_{\kappa(n,s)}^n|)^{\rho+2} ds \notag
\end{align}
which on the application of H\"{o}lder's inequality, Corollary \ref{cor:one-step:rate:uc} gives the following estimates, 
\begin{align}
H_4&\leq K\int_{\kappa(n,t)}^t \{E|x_s^n-x_{\kappa(n,s)}^n|^{2+\delta}\}^\frac{2}{2+\delta} \{E(1+|x_s^n|+|x_{\kappa(n,s)}^n|)^\frac{2\rho(2+\delta)}{\delta}\}^\frac{\delta}{2+\delta} ds \notag
\\
&\leq K n^{-1-\frac{2}{2+\delta}} \label{eq:H4}
\end{align}
for any $t\in [0,T]$ and $n\in\mathbb{N}$. For estimating $H_5$, one uses H\"{o}lder's equality, an elementary inequality of stochastic integrals, Remarks [\ref{rem:growth}, \ref{rem:growth(n)}] and Lemma \ref{lem:mb:scheme} to obtain,  
\begin{align}
H_5&:= K E\Big|\sum_{u=1}^{d}\sum_{j=1}^m \int_{\kappa(n,t)}^t \frac{\partial \sigma(x_s^n)}{\partial x^{u}} \int_{\kappa(n,s)}^s \sum_{k=1}^m \Lambda_1^{(n),k, uj}(x_{\kappa(n,r)}^n)dw_r^k dw_s^j\Big|^2 \notag
\\
&\leq K  \int_{\kappa(n,t)}^t \{E(1+|x_s^n|)^{\rho}\}^\frac{1}{2} \Big\{ E\Big|\int_{\kappa(n,s)}^s \sum_{k=1}^m \Lambda_1^{(n),k}(x_{\kappa(n,r)}^n)dw_r^k\Big|^4\Big\}^\frac{1}{2} ds \notag
\\
&\leq K  \int_{\kappa(n,t)}^t  \Big\{ n^{-1} E\int_{\kappa(n,s)}^s \Big|\sum_{k=1}^m \Lambda_1^{(n),k}(x_{\kappa(n,r)}^n)\Big|^4 dr\Big\}^\frac{1}{2} ds \leq Kn^{-2} \label{eq:H5}
\end{align}
for any $t\in[0,T]$ and $n\in\mathbb{N}$. Similarly, for estimating $H_6$, one writes, 
\begin{align}
H_6&:=K E\Big|\sum_{u=1}^{d}\sum_{j=1}^m \int_{\kappa(n,t)}^t \frac{\partial \sigma(x_s^n)}{\partial x^{u}} \int_{\kappa(n,s)}^s \int_Z \Lambda_2^{(n),uj}(x_{\kappa(n,r)}^n,z_1)\tilde{N}(dr,dz_1) dw_s^j\Big|^2 \notag
\\
&\leq K E\sum_{u=1}^{d}\sum_{j=1}^m \int_{\kappa(n,t)}^t \Big|\frac{\partial \sigma(x_s^n)}{\partial x^{u}}\Big|^2 \Big| \int_{\kappa(n,s)}^s \int_Z \Lambda_2^{(n),uj}(x_{\kappa(n,r)}^n,z_1)\tilde{N}(dr,dz_1)\Big|^2 ds \notag
\\
&\leq K E \int_{\kappa(n,t)}^t (1+|x_s^n|)^\rho \Big| \int_{\kappa(n,s)}^s \int_Z \Lambda_2^{(n)}(x_{\kappa(n,r)}^n,z_1)\tilde{N}(dr,dz_1)\Big|^2 ds \notag
\\
&\leq K  \int_{\kappa(n,t)}^t \Big\{E(1+|x_s^n|)^\frac{\rho(2+\delta)}{\delta}\Big\}^\frac{\delta}{2+\delta} \Big\{E\Big| \int_{\kappa(n,s)}^s \int_Z \Lambda_2^{(n)}(x_{\kappa(n,r)}^n,z_1)\tilde{N}(dr,dz_1)\Big|^{2+\delta}\Big\}^\frac{2}{2+\delta} ds \notag
\\
&\leq K  \int_{\kappa(n,t)}^t  \Big\{E\Big( \int_{\kappa(n,s)}^s \int_Z |\Lambda_2^{(n)}(x_{\kappa(n,r)}^n,z_1)|^2 \nu(dz_1)dr\Big)^\frac{2+\delta}{2}\notag
\\
&\qquad +E \int_{\kappa(n,s)}^s \int_Z |\Lambda_2^{(n)}(x_{\kappa(n,r)}^n,z_1)|^{2+\delta} \nu(dz_1)\Big\}^\frac{2}{2+\delta} ds \notag
\\
& \leq Kn^{-1-\frac{2}{2+\delta}} \label{eq:H6}
\end{align}
for any $t \in [0,T]$ and $n\in\mathbb{N}$. For estimating $H_7$, one proceeds as before,  
\begin{align}
H_7&:=  K E\Big|\sum_{u=1}^{d}\sum_{j=1}^m \int_{\kappa(n,t)}^t \frac{\partial \sigma(x_s^n)}{\partial x^{u}}  \int_{\kappa(n,s)}^s \int_Z \Lambda_3^{(n),uj}(x_{\kappa(n,r)}^n,z_1)N(dr,dz_1) dw_s^j\Big|^2 \notag
\\
&\leq K E \int_{\kappa(n,t)}^t (1+|x_s^n|)^\rho \Big|  \int_{\kappa(n,s)}^s \int_Z \Lambda_3^{(n)}(x_{\kappa(n,r)}^n,z_1)N(dr,dz_1)\Big|^2 ds \notag
\\
&\leq K E \int_{\kappa(n,t)}^t \Big\{E(1+|x_s^n|)^\frac{\rho(2+\delta)}{\delta}\Big\}^\frac{\delta}{2+\delta} \Big\{E\Big|  \int_{\kappa(n,s)}^s \int_Z \Lambda_3^{(n)}(x_{\kappa(n,r)}^n,z_1)N(dr,dz_1)\Big|^{2+\delta}\Big\}^\frac{2}{2+\delta} ds \notag
\\
&\leq K n^{-1-\frac{2}{2+\delta}}  \label{eq:H7}
\end{align}
for any $t\in [0,T]$  and $n\in\mathbb{N}$. Further, for estimating $H_8$, due to Assumption \ref{as:sde:lip:der} and Remark \ref{rem:growth}, one writes,
\begin{align}
H_8&:= K E\Big|\sum_{u=1}^{d} \int_{\kappa(n,t)}^t \int_Z \Big\{\frac{\partial \sigma(x_s^n)}{\partial x^{u}}-\frac{\partial \sigma(x_{\kappa(n,s)}^n)}{\partial x^{u}}\Big\} \gamma^{u}(x_{\kappa(n,s)}^n,z_1) \tilde{N}(ds,dz_1)\Big|^2 \notag
\\
&\leq K E \int_{\kappa(n,t)}^t  (1+|x_s^n|+|x_{\kappa(n,s)}^n|)^{\rho-2} |x_s^n-x_{\kappa(n,s)}^n|^2 (1+|x_{\kappa(n,s)}^n|)^2  ds \notag
\\
&\leq K  \int_{\kappa(n,t)}^t  \{E|x_s^n-x_{\kappa(n,s)}^n|^{2+\delta}\}^\frac{2}{2+\delta} \{E(1+|x_s^n|+|x_{\kappa(n,s)}^n|)^\frac{\rho (2+\delta)}{\delta}\}^\frac{\delta}{2+\delta}    ds \notag
\\
&\leq K n^{-1-\frac{2}{2+\delta}} \label{eq:H8}
\end{align} 
for any $t\in [0,T]$ and $n\in\mathbb{N}$.  Also, on applying Remark \ref{rem:growth} and Lemma \ref{lem:mb:scheme}, one gets
\begin{align}
H_9&:=K E\Big| \sum_{u=1}^{d} \int_{\kappa(n,t)}^t \int_Z \Big\{1 -\frac{1}{1+n^{-1}|x_{\kappa(n,s)}^n|^{4\rho}}\Big\} \frac{\partial \sigma(x_{\kappa(n,s)}^n)}{\partial x^{u}} \gamma^{u}(x_{\kappa(n,s)}^n,z_1) \tilde{N}(ds,dz_1)\Big|^2 \notag
\\
&\leq K n^{-2} E\Big| \sum_{u=1}^{d} \int_{\kappa(n,t)}^t \int_Z |x_{\kappa(n,s)}^n|^{8\rho} \Big|\frac{\partial \sigma(x_{\kappa(n,s)}^n)}{\partial x^{u}}\Big|^2 |\gamma^{u}(x_{\kappa(n,s)}^n,z_1)|^2 \nu(dz_1) ds  \notag
\\
& \leq K n^{-2} E \int_{\kappa(n,t)}^t  |x_{\kappa(n,s)}^n|^{4\rho} (1+|x_{\kappa(n,s)}^n|)^{\rho+2}  ds  \leq K n^{-3} \label{eq:H9} 
\end{align} 
for any $t\in[0,T]$  and $n\in\mathbb{N}$. Due to an elementary inequality of stochastic integrals, Remarks [\ref{rem:growth}, \ref{rem:growth(n)}],  H\"{o}lder's inequality and Lemma \ref{lem:mb:scheme}, one gets the following estimates of $H_{10}$,  
\begin{align}
H_{10}&:=K E\Big| \sum_{u=1}^{d} \int_{\kappa(n,t)}^t \int_Z \frac{\partial \sigma(x_s^n)}{\partial x^{u}} \int_{\kappa(n,s)}^s \sum_{k=1}^{m} \Gamma_1^{(n),k, u}(x_{\kappa(n,r)}^n, z_1)dw_r^k \tilde{N}(ds,dz_1)\Big|^2 \notag
\\
&\leq K E\int_{\kappa(n,t)}^t(1+|x_s^n|)^\rho  \int_Z  \Big| \int_{\kappa(n,s)}^s \sum_{k=1}^{m} \Gamma_1^{(n),k}(x_{\kappa(n,r)}^n, z_1)dw_r^k\Big|^2 \nu(dz_1)ds  \notag
\\
&\leq K E\int_{\kappa(n,t)}^t \{E(1+|x_s^n|)^{2\rho}\}^\frac{1}{2}  \int_Z \Big\{ E\Big| \int_{\kappa(n,s)}^s \sum_{k=1}^{m} \Gamma_1^{(n),k}(x_{\kappa(n,r)}^n, z_1)dw_r^k\Big|^4\Big\}^\frac{1}{2} \nu(dz_1)ds  \notag
\\
&\leq K E\int_{\kappa(n,t)}^t   \int_Z \Big\{n^{-1} E \int_{\kappa(n,s)}^s \Big|\sum_{k=1}^{m} \Gamma_1^{(n),k}(x_{\kappa(n,r)}^n, z_1)\Big|^4dr\Big\}^\frac{1}{2} \nu(dz_1)ds  \notag
\\
&\leq K n^{-2} \label{eq:H10}
\end{align}
for any $t \in [0,T]$  and $n\in\mathbb{N}$. Similarly, one also has the following estimate,   
\begin{align}
H_{11}&:=K E\Big| \sum_{u=1}^{d} \int_{\kappa(n,t)}^t \int_Z \frac{\partial \sigma(x_s^n)}{\partial x^{u}} \int_{\kappa(n,s)}^s \int_Z \Gamma_2^u(x_{\kappa(n,r)}^n,z_1,z_2) \tilde{N}(dr,dz_2) \tilde{N}(ds,dz_1) \Big|^2 \notag
\\
&\leq K E  \int_{\kappa(n,t)}^t \int_Z (1+ |x_s^n|)^\rho \Big|\int_{\kappa(n,s)}^s \int_Z \Gamma_2(x_{\kappa(n,r)}^n,z_1,z_2) \tilde{N}(dr,dz_2)\Big|^2 \nu(dz_1)ds  \notag
\\
&\leq K   \int_{\kappa(n,t)}^t \int_Z \{E(1+ |x_s^n|)^\frac{\rho(2+\delta)}{\delta}\}^\frac{\delta}{2+\delta} \Big\{E\Big|\int_{\kappa(n,s)}^s \int_Z \Gamma_2(x_{\kappa(n,r)}^n,z_1,z_2) \tilde{N}(dr,dz_2)\Big|^{2+\delta}\Big\}^\frac{2}{2+\delta} \nu(dz_1)ds  \notag
\\
&\leq K  \int_{\kappa(n,t)}^t \int_Z  \Big\{E\Big(\int_{\kappa(n,s)}^s \int_Z |\Gamma_2(x_{\kappa(n,r)}^n,z_1,z_2)|^2 \nu(dz_2) dr\Big)^\frac{2+\delta}{2}\notag
\\
&\qquad+E\int_{\kappa(n,s)}^s \int_Z |\Gamma_2(x_{\kappa(n,r)}^n,z_1,z_2)|^{2+\delta} \nu(dz_2) dr\Big\}^\frac{2}{2+\delta} \nu(dz_1)ds  \notag
\\
&\leq K n^{-1-\frac{2}{2+\delta}} \label{eq:H11}
\end{align}
for any $t\in[0,T]$ and $n\in\mathbb{N}$. Also,  
\begin{align}
&H_{12}:=K E\Big| \sum_{u=1}^{d} \int_{\kappa(n,t)}^t \int_Z \frac{\partial \sigma(x_s^n)}{\partial x^{u}} \int_{\kappa(n,s)}^s \int_Z \Gamma_3^u(x_{\kappa(n,r)}^n,z_1,z_2)N(dr,dz_2) \tilde{N}(ds,dz_1) \Big|^2\notag
\\
& \leq K E  \int_{\kappa(n,t)}^t \int_Z (1+|x_s^n|)^\rho \Big|\int_{\kappa(n,s)}^s \int_Z \Gamma_3(x_{\kappa(n,r)}^n,z_1,z_2)N(dr,dz_2)\Big|^2 \nu(dz_1) ds \notag
\\
& \leq K E  \int_{\kappa(n,t)}^t \int_Z \{E(1+|x_s^n|)^\frac{\rho(2+\delta)}{\delta}\}^\frac{\delta}{2+\delta} \Big\{E\Big|\int_{\kappa(n,s)}^s \int_Z \Gamma_3(x_{\kappa(n,r)}^n,z_1,z_2)N(dr,dz_2)\Big|^{2+\delta}\Big\}^\frac{2}{2+\delta} \nu(dz_1) ds \notag
\\
& \leq K n^{-1-\frac{2}{2+\delta}}  \label{eq:H12}
\end{align}
for any $t\in[0,T]$ and $n\in\mathbb{N}$. On the application of an elementary inequality of stochastic integrals, one obtains,
\begin{align}
H_{13}&:=K E\Big| \int_{\kappa(n,t)}^t \int_Z  \Big\{\sigma(x_s^n+\gamma^{(n)}(s,x_{\kappa(n,s)}^n,z_1))-\sigma(x_s^{n})-\sum_{u=1}^d \frac{\partial \sigma(x_s^n)}{\partial x^{u}} \gamma^{(n),u}(s,x_{\kappa(n,s)}^n,z_1)  \notag
\\
& - \frac{\sigma(x_s^n+\gamma(x_{\kappa(n,s)}^n,z_1))-\sigma(x_s^{n})-\sum_{u=1}^d \frac{\partial \sigma(x_{\kappa(n,s)}^n)}{\partial x^{u}} \gamma^{u}(x_{\kappa(n,s)}^n,z_1) }{1+n^{-1}|x_{\kappa(n,s)}^n|^{4\rho}}\Big\} N(ds,dz_1)\Big|^2 \notag
\\
&\leq K E \int_{\kappa(n,t)}^t \int_Z  \Big|\sigma(x_s^n+\gamma^{(n)}(s,x_{\kappa(n,s)}^n,z_1))-\sigma(x_s^{n})-\sum_{u=1}^d \frac{\partial \sigma(x_s^n)}{\partial x^{u}} \gamma^{(n),u}(s,x_{\kappa(n,s)}^n,z_1)  \notag
\\
& - \frac{\sigma(x_{\kappa(n,s)}^n+\gamma(x_{\kappa(n,s)}^n,z_1))-\sigma(x_{\kappa(n,s)}^n)-\sum_{u=1}^d \frac{\partial \sigma(x_{\kappa(n,s)}^n)}{\partial x^{u}} \gamma^{u}(x_{\kappa(n,s)}^n,z_1) }{1+n^{-1}|x_{\kappa(n,s)}^n|^{4\rho}}\Big|^2 \nu(dz_1) ds \notag
\\
&=K E \int_{\kappa(n,t)}^t \int_Z  \Big|\sigma(x_s^n+\gamma^{(n)}(s,x_{\kappa(n,s)}^n,z_1))-\sigma(x_{\kappa(n,s)}^n+\gamma(x_{\kappa(n,s)}^n,z_1)) \notag
\\
&+\sigma(x_{\kappa(n,s)}^n+\gamma(x_{\kappa(n,s)}^n,z_1))-\frac{\sigma(x_{\kappa(n,s)}^n+\gamma(x_{\kappa(n,s)}^n,z_1))}{1+n^{-1}|x_{\kappa(n,s)}^n|^{4\rho}} + \sigma(x_{\kappa(n,s)}^n)-\sigma(x_s^{n})\notag
\\
&+\frac{\sigma(x_{\kappa(n,s)}^n)}{1+n^{-1}|x_{\kappa(n,s)}^n|^{4\rho}}-\sigma(x_{\kappa(n,s)}^n) \notag
\\
& + \frac{\sum_{u=1}^d \frac{\partial \sigma(x_{\kappa(n,s)}^n)}{\partial x^{u}} \gamma^{u}(x_{\kappa(n,s)}^n,z_1) }{1+n^{-1}|x_{\kappa(n,s)}^n|^{4\rho}}-\sum_{u=1}^d \frac{\partial \sigma(x_{\kappa(n,s)}^n)}{\partial x^{u}} \gamma^{u}(x_{\kappa(n,s)}^n,z_1)\notag
\\
& + \sum_{u=1}^d \Big\{\frac{\partial \sigma(x_{\kappa(n,s)}^n)}{\partial x^{u}} -\frac{\partial \sigma(x_s^n)}{\partial x^{u}}\Big\} \gamma^{u}(x_{\kappa(n,s)}^n,z_1) \notag
\\
&-\sum_{u=1}^d\frac{\partial \sigma(x_s^n)}{\partial x^{u}}\int_{\kappa(n,s)}^s \sum_{k=1}^m \Gamma_1^{(n),j,u}(x_{\kappa(n,r)}^n,z_1)dw_r^k \notag
\\
&-\sum_{u=1}^d\frac{\partial \sigma(x_s^n)}{\partial x^{u}}\int_{\kappa(n,s)}^s \int_Z \Gamma_2^u(x_{\kappa(n,r)}^n,z_1,z_2)\tilde{N}(dr,dz_2) \notag
\\
&-\sum_{u=1}^d\frac{\partial \sigma(x_s^n)}{\partial x^{u}}\int_{\kappa(n,s)}^s \int_Z \Gamma_3^u(x_{\kappa(n,r)}^n,z_1,z_2)N(dr,dz_2)\Big|^2 \nu(dz_1) ds \notag
\end{align}
which can further be estimated by, 
\begin{align}
H_{13}&\leq K E \int_{\kappa(n,t)}^t \int_Z  |\sigma(x_s^n+\gamma^{(n)}(s,x_{\kappa(n,s)}^n,z_1))-\sigma(x_{\kappa(n,s)}^n+\gamma(x_{\kappa(n,s)}^n,z_1))|^2 \nu(dz_1)ds \notag
\\
&+ K E \int_{\kappa(n,t)}^t \int_Z \Big|\sigma(x_{\kappa(n,s)}^n+\gamma(x_{\kappa(n,s)}^n,z_1))-\frac{\sigma(x_{\kappa(n,s)}^n+\gamma(x_{\kappa(n,s)}^n,z_1))}{1+n^{-1}|x_{\kappa(n,s)}^n|^{4\rho}}\Big|^2\nu(dz_1)ds \notag
\\
&+ K E \int_{\kappa(n,t)}^t  |\sigma(x_{\kappa(n,s)}^n)-\sigma(x_s^{n})|^2 ds \notag
\\
&+ K E \int_{\kappa(n,t)}^t \int_Z\Big|\frac{\sigma(x_{\kappa(n,s)}^n)}{1+n^{-1}|x_{\kappa(n,s)}^n|^{4\rho}}-\sigma(x_{\kappa(n,s)}^n)\Big|^2 \nu(dz_1) ds \notag
\\
& + K E \int_{\kappa(n,t)}^t \int_Z\Big|\frac{\sum_{u=1}^d \frac{\partial \sigma(x_{\kappa(n,s)}^n)}{\partial x^{u}} \gamma^{u}(x_{\kappa(n,s)}^n,z_1) }{1+n^{-1}|x_{\kappa(n,s)}^n|^{4\rho}}-\sum_{u=1}^d \frac{\partial \sigma(x_{\kappa(n,s)}^n)}{\partial x^{u}} \gamma^{u}(x_{\kappa(n,s)}^n,z_1)\Big|^2 \nu(dz_1) ds\notag
\\
& + K E \int_{\kappa(n,t)}^t \int_Z\Big|\sum_{u=1}^d \Big\{\frac{\partial \sigma(x_{\kappa(n,s)}^n)}{\partial x^{u}} -\frac{\partial \sigma(x_s^n)}{\partial x^{u}}\Big\} \gamma^{u}(x_{\kappa(n,s)}^n,z_1)\Big|^2 \nu(dz_1) ds \notag
\\
&+K E \int_{\kappa(n,t)}^t \int_Z\Big|\sum_{u=1}^d\frac{\partial \sigma(x_s^n)}{\partial x^{u}}\int_{\kappa(n,s)}^s \sum_{k=1}^m \Gamma_1^{(n),j,u}(x_{\kappa(n,r)}^n,z_1)dw_r^k \Big|^2 \nu(dz_1) ds  \notag
\\
&+K E \int_{\kappa(n,t)}^t \int_Z\Big|\sum_{u=1}^d\frac{\partial \sigma(x_s^n)}{\partial x^{u}}\int_{\kappa(n,s)}^s \int_Z \Gamma_2^u(x_{\kappa(n,r)}^n,z_1,z_2)\tilde{N}(dr,dz_2) \Big|^2 \nu(dz_1) ds  \notag
\\
&+K E \int_{\kappa(n,t)}^t \int_Z\Big|\sum_{u=1}^d\frac{\partial \sigma(x_s^n)}{\partial x^{u}}\int_{\kappa(n,s)}^s \int_Z \Gamma_3^u(x_{\kappa(n,r)}^n,z_1,z_2)N(dr,dz_2)\Big|^2 \nu(dz_1) ds \notag
\end{align}
and then by applying Remarks [\ref{rem:growth}, \ref{rem:growth(n)}] and Assumption A-\ref{as:sde:lip:der}, one gets 
\begin{align}
H_{13}&\leq K E \int_{\kappa(n,t)}^t \int_Z (1+|x_s^n+\gamma^{(n)}(s,x_{\kappa(n,s)}^n,z_1)|+|x_{\kappa(n,s)}^n+\gamma(x_{\kappa(n,s)}^n,z_1)|)^\rho |x_s^n-x_{\kappa(n,s)}^n|^2 \nu(dz_1)ds\notag
\\
& + K E \int_{\kappa(n,t)}^t \int_Z (1+|x_s^n+\gamma^{(n)}(s,x_{\kappa(n,s)}^n,z_1)|+|x_{\kappa(n,s)}^n+\gamma(x_{\kappa(n,s)}^n,z_1)|)^\rho \notag
\\
& \times |\gamma^{(n)}(s,x_{\kappa(n,s)}^n,z_1)-\gamma(x_{\kappa(n,s)}^n,z_1)|^2 \nu(dz_1)ds \notag
\\
&+ K n^{-2} E \int_{\kappa(n,t)}^t \int_Z |x_{\kappa(n,s)}^n|^{8\rho}(1+|x_{\kappa(n,s)}^n+\gamma(x_{\kappa(n,s)}^n,z_1)|)^{\rho+2}\nu(dz_1)ds \notag
\\
&+ K E \int_{\kappa(n,t)}^t (1+|x_{\kappa(n,s)}^n|+|x_s^{n}|)^\rho |x_s^{n}-x_{\kappa(n,s)}^n|^2 ds \notag
\\
&+ K n^{-2} E \int_{\kappa(n,t)}^t |x_{\kappa(n,s)}^n|^{8\rho}(1+|x_{\kappa(n,s)}^n|)^{\rho+2}  ds \notag
\\
& + K E \int_{\kappa(n,t)}^t (1+|x_s^n|+|x_{\kappa(n,s)}^n|)^{\rho-2}|x_s^n-x_{\kappa(n,s)}^n|^2 (1+|x_{\kappa(n,s)}^n|)^2 ds \notag
\\
&+K E \int_{\kappa(n,t)}^t (1+|x_s^n|)^\rho \int_{\kappa(n,s)}^s (1+|x_{\kappa(n,r)}^n|)^{\rho+2}dr  ds  \notag
\\
&+K E \int_{\kappa(n,t)}^t (1+|x_s^n|)^\rho \int_{\kappa(n,s)}^s (1+|x_{\kappa(n,r)}^n|)^2 dr ds  \notag
\end{align}
for any $t\in[0,T]$ and $n\in\mathbb{N}$.  One can then simplify the above expression as below, 
\begin{align}
&H_{13}\leq K  \int_{\kappa(n,t)}^t \{E(1+|x_{\kappa(n,s)}^n|+|x_s^{n}|)^\frac{\rho(2+\delta)}{\delta}\}^\frac{\delta}{2+\delta} \{E|x_s^{n}-x_{\kappa(n,s)}^n|^{2+\delta}\}^\frac{2}{2+\delta} ds \notag
\\
&+ K E \int_{\kappa(n,t)}^t \int_Z (1+|\gamma^{(n)}(s,x_{\kappa(n,s)}^n,z_1)|+|\gamma(x_{\kappa(n,s)}^n,z_1)|)^\rho |x_s^n-x_{\kappa(n,s)}^n|^2 \nu(dz_1)ds\notag
\\
& + K E \int_{\kappa(n,t)}^t (1+|x_s^n|+|x_{\kappa(n,s)}^n|)^\rho \int_Z|\gamma^{(n)}(s,x_{\kappa(n,s)}^n,z_1)-\gamma(x_{\kappa(n,s)}^n,z_1)|^2 \nu(dz_1)ds \notag
\\
& + K E \int_{\kappa(n,t)}^t \int_Z (1+|\gamma^{(n)}(s,x_{\kappa(n,s)}^n,z_1)|+|\gamma(x_{\kappa(n,s)}^n,z_1)|)^\rho |\gamma^{(n)}(s,x_{\kappa(n,s)}^n,z_1)-\gamma(x_{\kappa(n,s)}^n,z_1)|^2 \nu(dz_1)ds \notag
\\
&+ K n^{-3} +K n^{-2} \notag
\end{align}
which can also be estimated by 
\begin{align}
H_{13}&\leq K  n^{-1-\frac{2}{2+\delta}} + K \int_{\kappa(n,t)}^t \{E |x_s^n-x_{\kappa(n,s)}^n|^{2+\delta}\}^\frac{2}{2+\delta}  \notag
\\
&\times \Big\{E\Big(\int_Z (1+|\gamma^{(n)}(s,x_{\kappa(n,s)}^n,z_1)|+|\gamma(x_{\kappa(n,s)}^n,z_1)|)^\rho  \nu(dz_1)\Big)^\frac{2+\delta}{\delta}\Big\}^\frac{\delta}{2+\delta}
 ds\notag
\\
& + K E \int_{\kappa(n,t)}^t (1+|x_s^n|+|x_{\kappa(n,s)}^n|)^\rho \int_Z|\gamma^{(n)}(s,x_{\kappa(n,s)}^n,z_1)-\gamma(x_{\kappa(n,s)}^n,z_1)|^2 \nu(dz_1)ds \notag
\\
& + K E \int_{\kappa(n,t)}^t \int_Z (1+|\gamma^{(n)}(s,x_{\kappa(n,s)}^n,z_1)|+|\gamma(x_{\kappa(n,s)}^n,z_1)|)^\rho |\gamma^{(n)}(s,x_{\kappa(n,s)}^n,z_1)-\gamma(x_{\kappa(n,s)}^n,z_1)|^2 \nu(dz_1)ds \notag
\end{align}
for any $t\in [0,T]$. Moreover, one obtains, 
\begin{align}
&H_{13}\leq K  n^{-1-\frac{2}{2+\delta}} \notag
\\
& + K E \int_{\kappa(n,t)}^t (1+|x_s^n|+|x_{\kappa(n,s)}^n|)^\rho \int_Z\Big|\int_{\kappa(n,s)}^s \sum_{k=1}^{m}\Gamma_1^{(n),k} (x_{\kappa(n,r)}^n,z_1)dw_r^k\Big|^2\nu(dz_1)ds \notag
\\
& + K E \int_{\kappa(n,t)}^t (1+|x_s^n|+|x_{\kappa(n,s)}^n|)^\rho \int_Z\Big|\int_{\kappa(n,s)}^s \int_Z \Gamma_2 (x_{\kappa(n,r)}^n,z_1,z_2)\tilde{N}(dr,dz_2)\Big|^2\nu(dz_1)ds \notag
\\
& + K E \int_{\kappa(n,t)}^t (1+|x_s^n|+|x_{\kappa(n,s)}^n|)^\rho \int_Z\Big|\int_{\kappa(n,s)}^s \int_Z \Gamma_3 (x_{\kappa(n,r)}^n,z_1,z_2)N(dr,dz_2)\Big|^2\nu(dz_1)ds \notag
\\
& + K E \int_{\kappa(n,t)}^t \int_Z (1+|\gamma^{(n)}(s,x_{\kappa(n,s)}^n,z_1)|+|\gamma(x_{\kappa(n,s)}^n,z_1)|)^\rho \Big|\int_{\kappa(n,s)}^s \sum_{k=1}^{m}\Gamma_1^{(n),k} (x_{\kappa(n,r)}^n,z_1)dw_r^k
\Big|^2 \nu(dz_1)ds \notag
\\
& + K E \int_{\kappa(n,t)}^t \int_Z (1+|\gamma^{(n)}(s,x_{\kappa(n,s)}^n,z_1)|+|\gamma(x_{\kappa(n,s)}^n,z_1)|)^\rho \Big|\int_{\kappa(n,s)}^s \int_Z \Gamma_2 (x_{\kappa(n,r)}^n,z_1,z_2)\tilde{N}(dr,dz_2)\Big|^2 \nu(dz_1)ds \notag
\\
& + K E \int_{\kappa(n,t)}^t \int_Z (1+|\gamma^{(n)}(s,x_{\kappa(n,s)}^n,z_1)|+|\gamma(x_{\kappa(n,s)}^n,z_1)|)^\rho \Big|\int_{\kappa(n,s)}^s \int_Z \Gamma_3 (x_{\kappa(n,r)}^n,z_1,z_2)N(dr,dz_2)\Big|^2 \nu(dz_1)ds \notag
\end{align}
and due to H\"{o}lder's inequality, 
\begin{align}
&H_{13}\leq K  n^{-1-\frac{2}{2+\delta}} + K  \int_{\kappa(n,t)}^t \{E(1+|x_s^n|+|x_{\kappa(n,s)}^n|)^{\frac{\rho(2+\delta)}{ \delta}}\}^\frac{\delta}{2+\delta} \notag 
\\
&\times \int_Z\Big\{E\Big|\int_{\kappa(n,s)}^s \sum_{k=1}^{m}\Gamma_1^{(n),k} (x_{\kappa(n,r)}^n,z_1)dw_r^k\Big|^{2+\delta}\Big\}^\frac{2}{2+\delta}\nu(dz_1)ds \notag
\\
& + K  \int_{\kappa(n,t)}^t \{E(1+|x_s^n|+|x_{\kappa(n,s)}^n|)^{\frac{\rho(2+\delta)}{ \delta}}\}^\frac{\delta}{2+\delta} \notag
\\
&\times \int_Z\Big\{E\Big|\int_{\kappa(n,s)}^s \int_Z \Gamma_2 (x_{\kappa(n,r)}^n,z_1,z_2)\tilde{N}(dr,dz_2)\Big|^{2+\delta}\Big\}^\frac{2}{2+\delta}\nu(dz_1)ds \notag
\\
& + K E \int_{\kappa(n,t)}^t \{E(1+|x_s^n|+|x_{\kappa(n,s)}^n|)^{\frac{\rho(2+\delta)}{ \delta}}\}^\frac{\delta}{2+\delta} \notag
\\
&\times \int_Z\Big\{E\Big|\int_{\kappa(n,s)}^s \int_Z \Gamma_3 (x_{\kappa(n,r)}^n,z_1,z_2)N(dr,dz_2)\Big|^{2+\delta}\Big\}^\frac{2}{2+\delta}\nu(dz_1)ds \notag
\\
& + K  \int_{\kappa(n,t)}^t \int_Z \{E(1+|\gamma^{(n)}(s,x_{\kappa(n,s)}^n,z_1)|+|\gamma(x_{\kappa(n,s)}^n,z_1)|)^{\frac{\rho(2+\delta)}{ \delta}}\}^\frac{\delta}{2+\delta} \notag
\\
&\times \Big\{E\Big|\int_{\kappa(n,s)}^s \sum_{k=1}^{m}\Gamma_1^{(n),k} (x_{\kappa(n,r)}^n,z_1)dw_r^k
\Big|^{2+\delta}\Big\}^\frac{2}{2+\delta} \nu(dz_1)ds \notag
\\
& + K  \int_{\kappa(n,t)}^t \int_Z \{E(1+|\gamma^{(n)}(s,x_{\kappa(n,s)}^n,z_1)|+|\gamma(x_{\kappa(n,s)}^n,z_1)|)^{\frac{\rho(2+\delta)}{ \delta}}\}^\frac{\delta}{2+\delta} \notag
\\
&\times \Big\{E\Big|\int_{\kappa(n,s)}^s \int_Z \Gamma_2 (x_{\kappa(n,r)}^n,z_1,z_2)\tilde{N}(dr,dz_2)\Big|^{2+\delta}\Big\}^\frac{2}{2+\delta} \nu(dz_1)ds \notag
\\
& + K  \int_{\kappa(n,t)}^t \int_Z \{E(1+|\gamma^{(n)}(s,x_{\kappa(n,s)}^n,z_1)|+|\gamma(x_{\kappa(n,s)}^n,z_1)|)^{\frac{\rho(2+\delta)}{ \delta}}\}^\frac{\delta}{2+\delta} \notag
\\
&\times \Big\{E\Big|\int_{\kappa(n,s)}^s \int_Z \Gamma_3 (x_{\kappa(n,r)}^n,z_1,z_2)N(dr,dz_2)\Big|^{2+\delta}\Big\}^\frac{2}{2+\delta} \nu(dz_1)ds \notag
\end{align}
which on the  application an elementary inequality of stochastic integrals, Remark \ref{rem:growth(n)}, Lemma \ref{lem:mb:scheme} yields, 
\begin{align}
&H_{13}\leq K  n^{-1-\frac{2}{2+\delta}} + K  \int_{\kappa(n,t)}^t  \int_Z\Big\{n^{-\frac{\delta}{2}}E\int_{\kappa(n,s)}^s \Big|\sum_{k=1}^{m}\Gamma_1^{(n),k} (x_{\kappa(n,r)}^n,z_1)\Big|^{2+\delta}dr\Big\}^\frac{2}{2+\delta}\nu(dz_1)ds \notag
\\
& + K  \int_{\kappa(n,t)}^t  \int_Z\Big\{n^{-\frac{\delta}{2}}E\int_{\kappa(n,s)}^s \int_Z |\Gamma_2 (x_{\kappa(n,r)}^n,z_1,z_2)|^{2+\delta}\nu(dz_2)dr\notag
\\
&+E\int_{\kappa(n,s)}^s \int_Z |\Gamma_2 (x_{\kappa(n,r)}^n,z_1,z_2)|^{2+\delta}\nu(dz_2)dr\Big\}^\frac{2}{2+\delta}\nu(dz_1)ds \notag
\\
& + K  \int_{\kappa(n,t)}^t \int_Z\Big\{n^{-\frac{\delta}{2}}E\int_{\kappa(n,s)}^s \int_Z |\Gamma_3 (x_{\kappa(n,r)}^n,z_1,z_2)|^{2+\delta} \nu(dz_2)dr \notag
\\
&+E\int_{\kappa(n,s)}^s \int_Z |\Gamma_3 (x_{\kappa(n,r)}^n,z_1,z_2)|^{2+\delta} \nu(dz_2)dr \notag
\\
& + n^{-1-\delta}E\int_{\kappa(n,s)}^s \int_Z |\Gamma_3 (x_{\kappa(n,r)}^n,z_1,z_2)|^{2+\delta}\nu(dz_2)dr\Big\}^\frac{2}{2+\delta}\nu(dz_1)ds \notag
\\
&\leq K  n^{-1-\frac{2}{2+\delta}}  \label{eq:H13}
\end{align}
for any $t\in [0,T]$ and $n\in\mathbb{N}$. Thus, by substituting values from equations \eqref{eq:H1} to \eqref{eq:H13} in equation \eqref{eq:H1+H13}, one completes the proof. 
\end{proof}
Let us first define, 
\begin{align*}
e_t^n:=x_t-x_t^n & =  \int_{0}^{t}\big\{b(x_s)-b^{(n)}(x_{\kappa(n,s)}^n)\big\}ds+\int_{0}^{t}\big\{\sigma(x_s)-\sigma^{(n)}(s,x_{\kappa(n,s)}^n)\big\}dw_s
\\
& + \int_{0}^{t}\int_Z \big\{\gamma(x_s,z)- \gamma^{(n)}(s,x_{\kappa(n,s)}^n,z)\big\}\tilde N(ds, dz)
\end{align*}
almost surely for any $t \in [0, T]$ and $n \in \mathbb{N}$. 
\begin{lem} \label{lem:b-b(n):rate}
Let Assumptions A-\ref{as:sde:ini} to A-\ref{as:sde:lip:der} be satisfied. Then, the following holds, 
\begin{align*}
E\int_0^t e_s^n (b(x_s^n) -  b^{(n)}(x_{\kappa(n,s)}^n))ds  \leq K \int_0^t \sup_{0 \leq r\leq s}E|e_s^n|^2ds+Kn^{-\frac{3}{2}-\frac{1}{2+\delta}}
\end{align*}
for any $t\in [0,T]$ and $n\in \mathbb{N}$ where positive constant $K$ does not depend on $n$. 
\end{lem}
\begin{proof} 
By the application of It\^{o}'s formula, 
\begin{align}
E\int_0^te_s^n\{b(x_s^n)&-b(x_{\kappa(n,s)}^n)\}ds=E\int_0^te_s^n\sum_{u=1}^{d}  \int_{\kappa(n,s)}^{s}\frac{\partial b(x_r^n)}{\partial x^{u}}b^{(n),u}(x_{\kappa(n,r)}^n)dr ds \notag
\\
&+\frac{1}{2}E\int_0^te_s^n\sum_{u}^{d} \sum_{v=1}^{d} \int_{\kappa(n,s)}^{s}\frac{\partial^2 b(x_r^n)}{\partial x^{u} \partial x^{v}} \sum_{j=1}^{m}\sigma^{(n),uj}(r,x_{\kappa(n,r)}^n) \sigma^{(n),vj}(r,x_{\kappa(n,r)}^n)dr ds \notag
\\
&+ E \int_0^te_s^n\sum_{u=1}^{d} \int_{\kappa(n,s)}^s \frac{\partial b(x_r^n)}{\partial x^{u}} \sum_{j=1}^m \sigma^{(n), uj}(r,x_{\kappa(n,r)}^n) dw_r^j ds \notag
\\
&+E \int_0^te_s^n \sum_{u=1}^{d} \int_{\kappa(n,s)}^s \int_Z \frac{\partial b(x_r^n)}{\partial x^{u}} \gamma^{(n), u}(r,x_{\kappa(n,r)}^n,z_1) \tilde{N}(dr,dz_1) ds \notag
\\
&+E \int_0^te_s^n\int_{\kappa(n,s)}^s \int_Z  \Big\{b(x_r^n+\gamma^{(n)}(r,x_{\kappa(n,r)}^n,z_1))-b(x_r^{n}) \notag
\\
&\qquad-\sum_{u=1}^d \frac{\partial b(x_r^n)}{\partial x^{u}} \gamma^{(n),u}(r,x_{\kappa(n,r)}^n,z_1) \Big\} \tilde{N}(dr,dz_1) ds \notag
\\
&+E \int_0^te_s^n\int_{\kappa(n,s)}^s \int_Z  \Big\{b(x_r^n+\gamma^{(n)}(r,x_{\kappa(n,r)}^n,z_1))-b(x_r^{n}) \notag
\\
&\qquad-\sum_{u=1}^d \frac{\partial b(x_r^n)}{\partial x^{u}} \gamma^{(n),u}(r,x_{\kappa(n,r)}^n,z_1) \Big\} \nu(dz_1)dr ds \label{eq:b:ito}
\end{align}
for any $t\in [0,T]$ and $n\in\mathbb{N}$. Let us now introduce a notation, 
\begin{align}
R(s,x_{\kappa(n,s)}^n):=&\sum_{u=1}^{d} \int_{\kappa(n,s)}^s \frac{\partial b(x_r^n)}{\partial x^{u}} \sum_{j=1}^m\sigma^{(n), uj}(r,x_{\kappa(n,r)}^n) dw_r^j \notag
\\
& +\sum_{u=1}^{d} \int_{\kappa(n,s)}^s \int_Z \frac{\partial b(x_r^n)}{\partial x^{u}} \gamma^{(n), u}(r,x_{\kappa(n,r)}^n,z_1) \tilde{N}(dr,dz_1)\notag
\\
&+\int_{\kappa(n,s)}^s \int_Z  \Big\{b(x_r^n+\gamma^{(n)}(r,x_{\kappa(n,r)}^n,z_1))-b(x_r^{n}) \notag
\\
& -\sum_{u=1}^d \frac{\partial b(x_r^n)}{\partial x^{u}} \gamma^{(n),u}(r,x_{\kappa(n,r)}^n,z_1) \Big\} \tilde{N}(dr,dz_1) \label{eq:R}
\end{align}
almost surely for any $s\in[0,T]$ and $n\in\mathbb{N}$. It is clear that $R(s,x_{\kappa(n,s)}^n)$ is an $\mathscr{F}_s$-measurable for every $s\in [0,T]$.  Then, one can write equation \eqref{eq:b:ito} as 
\begin{align}
E\int_0^te_s^n\{b(x_s^n)&-b(x_{\kappa(n,s)}^n)\}ds=E\int_0^te_s^n\sum_{u=1}^{d}  \int_{\kappa(n,s)}^{s}\frac{\partial b(x_r^n)}{\partial x^{u}}b^{(n),u}(x_{\kappa(n,r)}^n)dr ds \notag
\\
&+E\int_0^te_s^n\sum_{u}^{d} \sum_{v=1}^{d} \int_{\kappa(n,s)}^{s}\frac{\partial^2 b(x_r^n)}{\partial x^{u} \partial x^{v}} \sum_{j=1}^{m} \sigma^{(n),uj}(r,x_{\kappa(n,r)}^n) \sigma^{(n),vj}(r,x_{\kappa(n,r)}^n)dr ds \notag
\\
&+ E \int_0^te_s^n R(s,x_{\kappa(n,s)}^n) ds \notag
\\
&+E \int_0^te_s^n\int_{\kappa(n,s)}^s \int_Z  \Big\{b(x_r^n+\gamma^{(n)}(r,x_{\kappa(n,r)}^n,z_1))-b(x_r^{n}) \notag
\\
&\qquad-\sum_{u=1}^d \frac{\partial b(x_r^n)}{\partial x^{u}} \gamma^{(n),u}(r,x_{\kappa(n,r)}^n,z_1) \Big\} \nu(dz_1)dr ds 
\end{align} 
which can further be written as 
\begin{align}
E\int_0^te_s^n&\{b(x_s^n)-b(x_{\kappa(n,s)}^n)\}ds=E\int_0^te_s^n\sum_{u=1}^{d}  \int_{\kappa(n,s)}^{s}\frac{\partial b(x_r^n)}{\partial x^{u}}b^{(n),u}(x_{\kappa(n,r)}^n)dr ds \notag
\\
&+E\int_0^te_s^n\sum_{u}^{d} \sum_{v=1}^{d} \sum_{j=1}^{m}\int_{\kappa(n,s)}^{s}\frac{\partial^2 b(x_r^n)}{\partial x^{u} \partial x^{v}} \sigma^{(n),uj}(r,x_{\kappa(n,r)}^n) \sigma^{(n),vj}(r,x_{\kappa(n,r)}^n)dr ds \notag
\\
&+ E \int_0^te_{\kappa(n,s)}^n R(s,x_{\kappa(n,s)}^n) ds \notag
\\
&+ E \int_0^t\int_{\kappa(n,s)}^s \{\sigma(x_r)-\sigma(x_{r}^n)\}dw_r R(s,x_{\kappa(n,s)}^n) ds \notag
\\
&+ E \int_0^t\int_{\kappa(n,s)}^s \{\sigma(x_r^n)-\sigma^{(n)}(r,x_{\kappa(n,r)}^n)\}dw_r R(s,x_{\kappa(n,s)}^n) ds \notag
\\
&+ E \int_0^t\int_{\kappa(n,s)}^s \int_Z \{\gamma(x_r,z)-\gamma(x_r^n,z)\}\tilde{N}(dr,dz) R(s,x_{\kappa(n,s)}^n) ds \notag
\\
&+ E \int_0^t\int_{\kappa(n,s)}^s \int_Z \{\gamma(x_r^n,z)-\gamma^{(n)}(r,x_{\kappa(n,r)}^n,z)\}\tilde{N}(dr,dz) R(s,x_{\kappa(n,s)}^n) ds \notag
\\
&+E \int_0^te_s^n\int_{\kappa(n,s)}^s \int_Z  \Big\{b(x_r^n+\gamma^{(n)}(r,x_{\kappa(n,r)}^n,z_1))-b(x_r^{n}) \notag
\\
&\qquad-\sum_{u=1}^d \frac{\partial b(x_r^n)}{\partial x^{u}} \gamma^{(n),u}(r,x_{\kappa(n,r)}^n,z_1) \Big\} \nu(dz_1)dr ds \notag
\\
=:& \, E_1+E_2+E_3+E_4+E_5+E_6+E_7+E_8 \label{eq:E1+E8}
\end{align}
for any $t\in [0,T]$ and $n\in\mathbb{N}$. For estimating $E_1$, one uses Young's inequality, Remarks [\ref{rem:growth}, \ref{rem:growth(n)}] and Lemma \ref{lem:mb:scheme} to obtain the following estimates, 
\begin{align}
E_1&:= E\int_0^te_s^n\sum_{u=1}^{d}  \int_{\kappa(n,s)}^{s}\frac{\partial b(x_r^n)}{\partial x^{u}}b^{(n),u}(x_{\kappa(n,r)}^n)dr ds \notag
\\
&\leq K E\int_0^t |e_s^n|^2ds + K E\int_0^t\Big|\sum_{u=1}^{d}  \int_{\kappa(n,s)}^{s}\frac{\partial b(x_r^n)}{\partial x^{u}}b^{(n),u}(x_{\kappa(n,r)}^n)dr \Big|^2 ds \notag
\\
&\leq K\int_0^t \sup_{0 \leq r \leq s}E|e_r^n|^2ds + K n^{-2} \label{eq:E1}
\end{align}
for any $t\in [0,T]$ and $n\in\mathbb{N}$. Similarly, one uses Young's inequality, Remarks [\ref{rem:growth}, \ref{rem:growth(n)}] and Lemma \ref{lem:mb:scheme}, 
\begin{align}
E_2 &:= \frac{1}{2}E\int_0^te_s^n\sum_{u}^{d} \sum_{v=1}^{d} \sum_{j=1}^{m}\int_{\kappa(n,s)}^{s}\frac{\partial^2 b(x_r^n)}{\partial x^{u} \partial x^{v}} \sigma^{(n),uj}(r,x_{\kappa(n,r)}^n) \sigma^{(n),vj}(r,x_{\kappa(n,r)}^n)dr ds \notag
\\
&\leq K E\int_0^t |e_s^n|^2 ds \notag
\\
& + K E\int_0^t \Big|\sum_{u}^{d} \sum_{v=1}^{d} \sum_{j=1}^{m}\int_{\kappa(n,s)}^{s}\frac{\partial^2 b(x_r^n)}{\partial x^{u} \partial x^{v}} \sigma^{(n),uj}(r,x_{\kappa(n,r)}^n) \sigma^{(n),vj}(r,x_{\kappa(n,r)}^n)dr\Big|^2 ds \notag 
\\
&\leq K \int_0^t \sup_{0 \leq r \leq s}E|e_r^n|^2 ds + K n^{-2}\label{eq:E2} 
\end{align} 
for any $t\in [0,T]$ and $n\in\mathbb{N}$. One easily observes that $E_3$ can be estimated by,
\begin{align}
E_3&:=E\int_0^te_{\kappa(n,s)}^n R(s,x_{\kappa(n,s)}^n)ds = E\int_0^te_{\kappa(n,s)}^n E\big(R(s,x_{\kappa(n,s)}^n)\big|\mathscr{F}_{\kappa(n,s)}\big)ds =0 \label{eq:E3} 
\end{align} 
for any $t\in[0,T]$ and $n\in\mathbb{N}$. Furthermore, for estimating $E_4$, one proceeds as follows,  
\begin{align}
E_4 &:= E \int_0^t\int_{\kappa(n,s)}^s \{\sigma(x_r)-\sigma(x_{r}^n)\}dw_r R(s,x_{\kappa(n,s)}^n) ds \notag
\\
&=E \int_0^t\int_{\kappa(n,s)}^s \{\sigma(x_r)-\sigma(x_{r}^n)\}dw_r \sum_{u=1}^{d} \int_{\kappa(n,s)}^s \frac{\partial b(x_r^n)}{\partial x^{u}} \sum_{j=1}^m\sigma^{(n), uj}(r,x_{\kappa(n,r)}^n) dw_r^j ds \notag
\\
&+E \int_0^t\int_{\kappa(n,s)}^s \{\sigma(x_r)-\sigma(x_{r}^n)\}dw_r \sum_{u=1}^{d} \int_{\kappa(n,s)}^s \int_Z \frac{\partial b(x_r^n)}{\partial x^{u}} \gamma^{(n), u}(r,x_{\kappa(n,r)}^n,z_1) \tilde{N}(dr,dz_1) ds\notag
\\
&+E \int_0^t\int_{\kappa(n,s)}^s \{\sigma(x_r)-\sigma(x_{r}^n)\}dw_r \int_{\kappa(n,s)}^s \int_Z  \Big\{b(x_r^n+\gamma^{(n)}(r,x_{\kappa(n,r)}^n,z_1))-b(x_r^{n}) \notag
\\
& \qquad -\sum_{u=1}^d \frac{\partial b(x_r^n)}{\partial x^{u}} \gamma^{(n),u}(r,x_{\kappa(n,r)}^n,z_1) \Big\} \tilde{N}(dr,dz_1) ds \notag
\end{align} 
for any $t\in[0,T]$ and $n\in\mathbb{N}$. Notice that processes $w$ and $N$ are independent, hence second and third terms on the right hand side of the above equation are zero.  Thus, one obtains,  
\begin{align}
E_4&\leq K E \int_0^t\int_{\kappa(n,s)}^s |\sigma(x_r)-\sigma(x_{r}^n)| \sum_{u=1}^{d}\sum_{j=1}^m  \Big|\frac{\partial b(x_r^n)}{\partial x^{u}}\Big| |\sigma^{(n), uj}(r,x_{\kappa(n,r)}^n)| dr ds \notag
\end{align}
which on using Remarks [\ref{rem:growth}, \ref{rem:growth(n)}], Young's inequality, H\"{o}lder's inequality and Lemma \ref{lem:mb:scheme} gives, 
\begin{align}
E_4& \leq K E \int_0^t\int_{\kappa(n,s)}^s (1+|x_r|+|x_{r}^n|)^\frac{\rho}{2}|x_r-x_{r}^n| (1+|x_r^n|)^\rho |\sigma^{(n)}(r,x_{\kappa(n,r)}^n)| dr ds \notag
\\
&\leq K n^{2\alpha} E \int_0^t\int_{\kappa(n,s)}^s |x_r-x_{r}^n|^2  dr ds\notag
\\
&+K n^{-2\alpha} E \int_0^t\int_{\kappa(n,s)}^s (1+|x_r|+|x_{r}^n|)^\rho (1+|x_r^n|)^{2\rho} |\sigma^{(n)}(r,x_{\kappa(n,r)}^n)|^2 dr ds  \notag
\\
&\leq K n^{2\alpha-1}  \int_0^t \sup_{0 \leq r \leq s}E|e_{r}^n|^2 ds+K n^{-2\alpha-1}   \notag
\end{align}
and then one takes $\alpha=1/2$ to obtain the following estimates, 
\begin{align}
E_4 \leq K   \int_0^t \sup_{0 \leq r \leq s}E|e_{r}^n|^2 ds+K n^{-2}   \label{eq:E4}
\end{align}
for any $t\in[0,T]$ and $n\in\mathbb{N}$. Before estimating rest of the terms on the right hand side of equation \eqref{eq:E1+E7}, one proceeds as follows, 
\begin{align}
E|R(s,x_{\kappa(n,s)}^n)|^2 &\leq K E\Big|\sum_{u=1}^{d} \int_{\kappa(n,s)}^s \frac{\partial b(x_r^n)}{\partial x^{u}} \sum_{j=1}^m \sigma^{(n), uj}(r,x_{\kappa(n,r)}^n) dw_r^j\Big|^2 \notag
\\
& + E\Big|\sum_{u=1}^{d} \int_{\kappa(n,s)}^s \int_Z \frac{\partial b(x_r^n)}{\partial x^{u}} \gamma^{(n), u}(r,x_{\kappa(n,r)}^n,z_1) \tilde{N}(dr,dz_1)\Big|^2 \notag
\\
&+K E\Big|\int_{\kappa(n,s)}^s \int_Z  \Big\{b(x_r^n+\gamma^{(n)}(r,x_{\kappa(n,r)}^n,z_1))-b(x_r^{n})\notag
\\
&-\sum_{u=1}^d \frac{\partial b(x_r^n)}{\partial x^{u}} \gamma^{(n),u}(r,x_{\kappa(n,r)}^n,z_1) \Big\} \tilde{N}(dr,dz_1)\Big|^2 \notag
\end{align}
which on the application of an elementary inequality of stochastic integral gives 
\begin{align}
E|R(s,x_{\kappa(n,s)}^n)|^2 &\leq K E\sum_{u=1}^{d}\sum_{j=1}^m \int_{\kappa(n,s)}^s \Big|\frac{\partial b(x_r^n)}{\partial x^{u}}\Big|^2 |\sigma^{(n), uj}(r,x_{\kappa(n,r)}^n)|^2 dr \notag
\\
&+ E\sum_{u=1}^{d} \int_{\kappa(n,s)}^s \int_Z \Big|\frac{\partial b(x_r^n)}{\partial x^{u}} \Big|^2 |\gamma^{(n), u}(r,x_{\kappa(n,r)}^n,z_1)|^2 \nu(dz_1) dr \notag
\\
&+KE\int_{\kappa(n,s)}^s \int_Z  \Big|b(x_r^n+\gamma^{(n)}(r,x_{\kappa(n,r)}^n,z_1))-b(x_r^{n})\notag
\\
&-\sum_{u=1}^d \frac{\partial b(x_r^n)}{\partial x^{u}} \gamma^{(n),u}(r,x_{\kappa(n,r)}^n,z_1)\Big|^2 \nu(dz_1)dr  \notag
\end{align}
and then due to Remarks [\ref{rem:growth}, \ref{rem:growth(n)}], Lemma \ref{lem:mb:scheme}, one obtains
\begin{align}
E|R(s,x_{\kappa(n,s)}^n)|^2 &\leq K E \int_{\kappa(n,s)}^s (1+|x_r^n|)^{2\rho}  |\sigma^{(n)}(r,x_{\kappa(n,r)}^n)|^2 dr \notag
\\
&+ E \int_{\kappa(n,s)}^s \int_Z (1+|x_r^n|)^{2\rho} |\gamma^{(n)}(r,x_{\kappa(n,r)}^n,z_1)|^2 \nu(dz_1) dr \notag
\\
+KE\int_{\kappa(n,s)}^s \int_Z  &(1+|x_r^n+\gamma^{(n)}(r,x_{\kappa(n,r)}^n,z_1)|+|x_r^{n}|)^{2\rho-2} |\gamma^{(n)}(r,x_{\kappa(n,r)}^n,z_1)|^4 \nu(dz_1)dr  \notag
\\
&\leq K n^{-1} \label{eq:R:rate}
\end{align}
for any $s\in[0,T]$ and $n\in\mathbb{N}$.  For estimating $E_5$, one applies H\"{o}lder's inequality, an elementary inequality of stochastic integrals and Lemma \ref{lem:sigma-sigma(n):rate} to obtain the following estimates,  
\begin{align}
E_5&:=E \int_0^t\int_{\kappa(n,s)}^s \{\sigma(x_r^n)-\sigma^{(n)}(r,x_{\kappa(n,r)}^n)\}dw_r R(s,x_{\kappa(n,s)}^n) ds \notag
\\
&\leq  \int_0^t \Big\{E\Big|\int_{\kappa(n,s)}^s \{\sigma(x_r^n)-\sigma^{(n)}(r,x_{\kappa(n,r)}^n)\}dw_r\Big|^2\Big\}^\frac{1}{2} \{E|R(s,x_{\kappa(n,s)}^n)|^2\}^\frac{1}{2} ds \notag 
\\
&\leq  \int_0^t \Big\{E\int_{\kappa(n,s)}^s |\sigma(x_r^n)-\sigma^{(n)}(r,x_{\kappa(n,r)}^n)|^2 dr \Big\}^\frac{1}{2} \{E|R(s,x_{\kappa(n,s)}^n)|^2\}^\frac{1}{2} ds \notag
\\
&\leq Kn^{-\frac{3}{2}-\frac{1}{2+\delta}} \label{eq:E5}  
\end{align}
for any $t\in [0,T]$ and $n\in\mathbb{N}$. Moreover, for estimating $E_6$ and $E_7$, one uses H\"{o}lder's inequality to obtain the following estimates, 
\begin{align}
E_6+E_7&:= E \int_0^t\int_{\kappa(n,s)}^s \int_Z \{\gamma(x_r,z)-\gamma(x_r^n,z)\}\tilde{N}(dr,dz) R(s,x_{\kappa(n,s)}^n) ds \notag
\\
&+ E \int_0^t\int_{\kappa(n,s)}^s \int_Z \{\gamma(x_r^n,z)-\gamma^{(n)}(r,x_{\kappa(n,r)}^n,z)\}\tilde{N}(dr,dz) R(s,x_{\kappa(n,s)}^n) ds \notag
\\
& \leq  \int_0^t \Big\{E\Big|\int_{\kappa(n,s)}^s \int_Z \{\gamma(x_r,z)-\gamma(x_r^n,z)\}\tilde{N}(dr,dz)\Big|^2\Big\}^\frac{1}{2}  \{E|R(s,x_{\kappa(n,s)}^n)|^2\}^\frac{1}{2} ds \notag
\\
&+  \int_0^t\Big\{E\Big|\int_{\kappa(n,s)}^s\int_Z \{\gamma(x_r^n,z)-\gamma^{(n)}(r,x_{\kappa(n,r)}^n,z)\}\tilde{N}(dr,dz)\Big|^2\Big\}^\frac{1}{2} \{E|R(s,x_{\kappa(n,s)}^n)|^2\}^\frac{1}{2} ds \notag
\end{align} 
which on the application of an elementary inequality of stochastic integral yields, 
\begin{align}
E_6+E_7 & \leq  \int_0^t \Big\{E\int_{\kappa(n,s)}^s \int_Z|\gamma(x_r,z)-\gamma(x_r^n,z)|^2 \nu(dz)dr\Big\}^\frac{1}{2}  \{E|R(s,x_{\kappa(n,s)}^n)|^2\}^\frac{1}{2} ds \notag
\\
&+  \int_0^t\Big\{E\int_{\kappa(n,s)}^s \int_Z|\gamma(x_r^n,z)-\gamma^{(n)}(r,x_{\kappa(n,r)}^n,z)|^2 \nu(dz) dr\Big\}^\frac{1}{2} \{E|R(s,x_{\kappa(n,s)}^n)|^2\}^\frac{1}{2} ds \notag
\end{align} 
and then one uses Remarks [\ref{rem:growth}, \ref{rem:growth(n)}], equation \eqref{eq:R:rate}, Lemma \ref{lem:gamma-gamma(n):rate} and Young's inequality to obtain, 
\begin{align}
E_6+E_7 & \leq  K n^{-1}\int_0^t \Big\{\sup_{0\leq r\leq s}E |e_r^n|^2\Big\}^\frac{1}{2}  ds +Kn^{-2} \notag
\\
&\leq  K \int_0^t \sup_{0\leq r\leq s}E |e_r^n|^2  ds + Kn^{-2} \label{eq:E6+E7}
\end{align} 
for any $t\in[0,T]$ and $n\in\mathbb{N}$. For estimating $E_8$, one uses Young's inequality to obtain the following,  
\begin{align}
E_8&:=E \int_0^te_s^n\int_{\kappa(n,s)}^s \int_Z  \Big\{b(x_r^n+\gamma^{(n)}(r,x_{\kappa(n,r)}^n,z_1))-b(x_r^{n}) \notag
\\
&\qquad-\sum_{u=1}^d \frac{\partial b(x_r^n)}{\partial x^{u}} \gamma^{(n),u}(r,x_{\kappa(n,r)}^n,z_1) \Big\} \nu(dz_1)dr ds \notag
\\
&\leq K E \int_0^t|e_s^n|^2ds+K n^{-1}E \int_0^t\int_{\kappa(n,s)}^s \int_Z  \Big|b(x_r^n+\gamma^{(n)}(r,x_{\kappa(n,r)}^n,z_1))-b(x_r^{n}) \notag
\\
&\qquad-\sum_{u=1}^d \frac{\partial b(x_r^n)}{\partial x^{u}} \gamma^{(n),u}(r,x_{\kappa(n,r)}^n,z_1) \Big|^2 \nu(dz_1)dr ds \notag
\\
&\leq  K \int_0^t \sup_{0\leq r \leq s} E|e_r^n|^2ds + K n^{-1}E \int_0^t\int_{\kappa(n,s)}^s \int_Z (1+|x_r^n+\gamma^{(n)}(r,x_{\kappa(n,r)}^n,z_1)|+|x_r^n|)^{2\rho-2} \notag
\\
&\qquad \times |\gamma^{(n)}(r,x_{\kappa(n,r)}^n,z_1)|^4 \nu(dz_1)dr ds \notag
\end{align}
which further implies, 
\begin{align}
E_8\leq  K \int_0^t \sup_{0\leq r \leq s} E|e_r^n|^2ds+Kn^{-2} \label{eq:E8}
\end{align}
for any $t\in[0,T]$ and $n\in\mathbb{N}$. The proof is completed by substituting values from equations \eqref{eq:E1} to \eqref{eq:E8} in equation \eqref{eq:E1+E8}.  
\end{proof}
\begin{proof}[\textbf{Proof of Theorem \ref{thm:main:thm}}]
By using It\^{o}'s formula,
\begin{align*}
|e_t^n|^2& = 2 \int_{0}^{t} e_s^n\{b(x_s)-b^{(n)}(x_{\kappa(n,s)}^n)\} ds + 2 \int_{0}^{t} e_s^n\{\sigma(x_s)-\sigma^{(n)}(s,x_{\kappa(n,s)}^n)\} dw_s
\\
&  +   \int_{0}^{t} |\sigma(x_s)-\sigma^{(n)}(s,x_{\kappa(n,s)}^n)|^2 ds\notag
\\
&+ 2 \int_{0}^{t} \int_Z e_s^n  \{\gamma(x_s,z)- \gamma^{(n)}(s,x_{\kappa(n,s)}^n,z)\}\tilde{N}(ds, dz)
\\
& +  \int_{0}^{t} \int_Z \big\{  |e_s^n+\gamma(x_s,z)- \gamma^{(n)}(s,x_{\kappa(n,s)}^n,z)|^2 -|e_s^n|^2\notag
\\
&\qquad-2e_s^n\{\gamma(x_s,z)- \gamma^{(n)}(s,x_{\kappa(n,s)}^n,z)\}\big\} N(ds, dz)
\end{align*}
almost surely for any $t \in [0,T]$ and $n\in\mathbb{N}$. Thus,
\begin{align}
E|e_t^n|^2& = 2 \int_{0}^{t} e_s^n\{b(x_s)-b^{(n)}(x_{\kappa(n,s)}^n)\} ds +  E \int_{0}^{t} |\sigma(x_s)-\sigma^{(n)}(s,x_{\kappa(n,s)}^n)|^2 ds \notag 
\\
& +  E \int_{0}^{t} \int_Z \big\{  |e_s^n+\gamma(x_s,z)- \gamma^{(n)}(s,x_{\kappa(n,s)}^n,z)|^2 -|e_s^n|^2\notag
\\
&-2e_s^n\{\gamma(x_s,z)- \gamma^{(n)}(s,x_{\kappa(n,s)}^n,z)\}\big\}  \nu(dz) ds \notag
\end{align}
which can further be written as 
\begin{align}
E|e_t^n|^2& = 2 \int_{0}^{t} e_s^n\{b(x_s)-b(x_s^n)\} ds+2 \int_{0}^{t} e_s^n\{b(x_s^n)-b^{(n)}(x_{\kappa(n,s)}^n)\} ds \notag
\\
& +  E \int_{0}^{t} |\sigma(x_s)-\sigma(x_s^n)|^2 ds+  E \int_{0}^{t} |\sigma(x_s^n)-\sigma^{(n)}(s,x_{\kappa(n,s)}^n)|^2 ds \notag 
\\
&+2E\int_{0}^{t}\sum_{i=1}^d \sum_{j=1}^m \{\sigma^{ij}(x_s)-\sigma^{ij}(x_s^n)\}\{\sigma^{ij}(x_s^n)-\sigma^{(n),ij}(s,x_{\kappa(n,s)}^n)\} ds \notag
\\
& + E \int_{0}^{t} \int_Z   |\gamma(x_s,z)- \gamma^{(n)}(s,x_{\kappa(n,s)}^n,z)|^2 \nu(dz) ds \notag
\end{align}
and then one uses Young's inequality to obtain the following estimate,
\begin{align}
E|e_t^n|^2& \leq  \int_{0}^{t} \big(2e_s^n\{b(x_s)-b(x_s^n)\} + \eta |\sigma(x_s)-\sigma(x_s^n)|^2 \big) ds \notag
\\
&+2 \int_{0}^{t} e_s^n\{b(x_s^n)-b^{(n)}(x_{\kappa(n,s)}^n)\} ds \notag
\\
& +  K E \int_{0}^{t} |\sigma(x_s^n)-\sigma^{(n)}(s,x_{\kappa(n,s)}^n)|^2 ds \notag 
\\
& + K E \int_{0}^{t} \int_Z   |\gamma(x_s,z)- \gamma(x_s^n,z)|^2 \nu(dz) ds  \notag
\\
&+  K E \int_{0}^{t} \int_Z   |\gamma(x_s^n,z)- \gamma^{(n)}(s,x_{\kappa(n,s)}^n,z)|^2 \nu(dz) ds \notag
\end{align}
for any $t\in[0,T]$ and $n\in\mathbb{N}$. Further, by using Assumption A-\ref{as:sde:lipschitz}, Lemmas [\ref{lem:gamma-gamma(n):rate}, \ref{lem:sigma-sigma(n):rate}, \ref{lem:b-b(n):rate}], following can be obtained, 
\begin{align}
E|e_t^n|^2& \leq  \int_{0}^{t} \sup_{0\leq r\leq s} E|e_r^n|^2 ds +Kn^{-1-\frac{2}{2+\delta}}+Kn^{-\frac{3}{2}-\frac{1}{2+\delta}} +Kn^{-2} \notag
\end{align}
for any $t\in [0,T]$ and $n\in\mathbb{N}$. Finally, the application of Gronwall's inequality completes the proof.  
\end{proof}


\end{document}